\documentclass[reqno]{amsart}

\usepackage[T1]{fontenc}
\usepackage[utf8]{inputenc}
\usepackage[USenglish]{babel}

\usepackage{amsmath,amsthm,amssymb,dsfont,upgreek}
\usepackage{graphics,xcolor}
\usepackage[style=numeric,giveninits=true,isbn=false,url=false]{biblatex}
\usepackage[hidelinks]{hyperref}
\usepackage{enumitem}
\usepackage{tikz,pgfplots,tkz-euclide}
\usetikzlibrary{intersections,arrows.meta,patterns.meta,calc}
\pgfplotsset{compat=1.11}
\usepackage{nicefrac}
\usepackage[sort,capitalize,nameinlink]{cleveref}
\usepackage[babel]{csquotes}
\usepackage{centernot}

\definecolor{dred}{HTML}{C11B17}
\definecolor{dgreen}{HTML}{41A317}

\newcommand{\cemph}[1]{\emph{\color{dred}#1}}
\newcommand{\R}{\ensuremath{\mathds{R}}}
\newcommand{\B}{\ensuremath{\mathds{B}}}
\newcommand{\N}{\ensuremath{\mathds{N}}}
\newcommand{\C}{\ensuremath{\mathds{C}}}
\newcommand{\Dcal}{\ensuremath{\mathcal{D}}}
\newcommand{\Bcal}{\ensuremath{\mathcal{B}}}
\newcommand{\Jcal}{\ensuremath{\mathcal{J}}}
\newcommand{\Lcal}{\ensuremath{\mathcal{L}}}
\newcommand{\Scal}{\ensuremath{\mathcal{S}}}
\newcommand{\Pcal}{\ensuremath{\mathcal{P}}}
\newcommand{\CK}{\ensuremath{\mathcal{K}}}
\newcommand{\eps}{\ensuremath\varepsilon}

\newcommand\gauge[1]{\left\| #1 \right\|}

\newcommand\abs[1]{\left\vert#1\right\vert}
\newcommand{\setn}[1]{\left\{#1\right\}}
\newcommand{\setcond}[2]{\left\{#1 \,:\, #2\right\}}
\newcommand{\defeq}{\mathrel{\mathop:}=}

\newcommand{\ii}{\mathrm{i}}
\newcommand{\gr}[2]{\mathrm{Gr}_{#1}(#2)}

\newcommand{\para}{nontrivial}

\newcommand\John{\mathcal{E}_J}
\DeclareMathOperator{\trace}{trace}
\DeclareMathOperator{\vol}{vol}
\DeclareMathOperator{\conv}{conv}
\DeclareMathOperator{\aff}{aff}
\DeclareMathOperator{\bd}{bd}
\DeclareMathOperator{\inte}{int}
\DeclareMathOperator{\gl}{GL}

\newtheorem{theorem}{Theorem}[section]
\newtheorem{lemma}[theorem]{Lemma}
\newtheorem{remark}[theorem]{Remark}
\newtheorem{proposition}[theorem]{Proposition}
\newtheorem{corollary}[theorem]{Corollary}

\title[Minkowski Chirality]{Minkowski Chirality: \\A Measure of Reflectional Asymmetry of Convex Bodies}

\author[A. Caragea]{Andrei Caragea}
\author[K. von Dichter]{Katherina von Dichter}
\author[K. Gottwald]{Kurt Klement Gottwald}
\author[F. Grundbacher]{Florian Grundbacher}
\author[T. Jahn]{Thomas Jahn}
\author[M. Runge]{Mia Runge}

\keywords{asymmetry, chirality, convex body, optimal containment, reflection}
\subjclass{52A40, 52A20}
\date{\today}

\begin{document}
\parindent 0pt

\begin{abstract}
Using an optimal containment approach, we quantify the asymmetry of convex bodies in $\R^n$ with respect to reflections across affine subspaces of a given dimension.
We prove general inequalities relating these \enquote{Minkowski chirality} measures to Banach--Mazur distances and to each other, and prove their continuity with respect to the Hausdorff distance. 
In the planar case, we determine the reflection axes at which the Minkowski chirality of triangles and parallelograms is attained, and show that $\sqrt{2}$ is a tight upper bound on the chirality in both cases.
\end{abstract}
 
\maketitle

\section{Introduction}

In his ground-laying work \cite{Grunbaum1963}, Grünbaum set up a general framework for quantifying the point (a)symmetry of convex bodies, i.e., compact convex sets with nonempty interior.
Specifically, a measure of (a)symmetry is a similarity-invariant (or even affinely invariant) Hausdorff continuous function $f$ that takes convex bodies to the unit interval with the property that $f(K)=1$ if and only if $K$ is point-symmetric.
In \cite{Grunbaum1963}, some generalizations are discussed, for example quantifying (a)symmetry with respect to reflections across affine subspaces of dimension at least one.
However, the author mentions lack of results in the literature in this direction.
Different notions of chirality or axiality for quantifying the (a)symmetry of planar shapes with respect to reflections across straight lines have been investigated in the mathematical literature in the past decades \cite{ChakerianStein1965,FaryRedei1950,GoenkaMoore2023,Valcourt1966b,Valcourt1966a}.
Asymmetry notions for planar convex bodies are also studied in mathematical chemistry \cite{BudaHeyde1992,BudaMislow1991,BudaMislow1992,WeinbergMislow1993}, where polygons serve as abstractions of molecules and where chirality impacts chemical properties.

Our contribution is based on an extension of the notion of Minkowski asymmetry \cite{BrandenbergGonzalez2017,BrandenbergKonig2015}, which, for a convex body $K$, is defined as the smallest dilation factor $\lambda>0$ such that $K$ is a subset of a translated and dilated copy of $-K$, the mirror image of $K$ upon reflection across the coordinate origin.
We incorporate reflections across higher-dimensional (affine) subspaces by defining the
\cemph{$j$th Minkowski chirality} $\alpha_j(K)$ as the smallest dilation factor $\lambda>0$ such that the convex body $K\subset\R^n$ is a subset of a translated and dilated copy of $\Phi_U(K)$, where $\Phi_U$ denotes the \cemph{reflection} across the $j$-dimensional affine subspace $U\subset\R^n$ for $j\in\setn{0,\ldots,n}$.
Note that the Minkowski asymmetry is $\alpha_0(K)$ in this terminology.

It is well-known that $\alpha_0(K) \in [1,n]$ for all convex bodies $K \subset \R^n$, with $\alpha_0(K)=1$ if and only if $K$ is point-symmetric, and $\alpha_0(K)=n$ if and only if $K$ is a fulldimensional simplex, see \cite{Grunbaum1963} and \cite[Note~14 for Section~3.1]{Schneider2014}.
Our main result for convex bodies in general dimensions extends the upper bound on the Minkowski asymmetry to all Minkowski chiralities $\alpha_j(K)$ for any $j \in \setn{0,\ldots,n}$.

\begin{theorem}\label{thm:dbm}
Let $K \subset\R^n$ be a convex body and $j \in \setn{0,\ldots,n}$.
Then
\begin{equation*}
    1 \leq \alpha_j(K) \leq \min \setn{ n,\frac{\alpha_0(K)+1}{2} \sqrt{n}  },
\end{equation*}
with $\alpha_j(K)=1$ if and only if there exists a $j$-dimensional affine subspace $U$ such that $K=\Phi_U(K)$.
\end{theorem}

In fact, the upper bound in \cref{thm:dbm} can be strengthened and unified to
\begin{equation}\label{eq:improved_upperbound}
   \alpha_j(K) \leq \sqrt{\alpha_0(K)n }.
\end{equation}
Since $\alpha_0(K) \leq n$ with $\alpha_0(K) = n$ solely for simplices, this result implies $\alpha_j(K) \leq n$ and in particular that only simplices might have $j$th Minkowski chirality $n$.
The inequality \eqref{eq:improved_upperbound} is based on a bound on the Banach--Mazur distance $d_{BM}(K,\B_2^n) \leq \sqrt{ \alpha_0(K)n}$ for any convex body $K \subset \R^n$ from an unpublished manuscript \cite{BrandenbergGonzalezGrundbacher2025}.

We recall that the \cemph{Banach--Mazur distance} between convex bodies $K,L \subset \R^n$ is defined by 
\begin{equation*}
    d_{BM}(K,L)=\inf\{\lambda > 0:t^1 + K\subset A(L)\subset t^2+\lambda K,\,A\in\gl(\R^n), \, t^1,t^2\in\R^n\},
\end{equation*}
where $\gl(\R^n)$ denotes the set of invertible real $n\times n$ matrices, see \cite[p.~589]{Schneider2014}.

The inequality \eqref{eq:improved_upperbound} is also consequential for the absolute upper bound on the $j$th Minkowski chirality.
Based on a stability result from \cite{Schneider2009}, any convex body $K$ with Minkowski asymmetry $\alpha_0(K)$ near $n$ is close to a simplex in the Banach-Mazur distance.
Together with \eqref{eq:improved_upperbound}, this means that either the supremum of $\alpha_j(T)$ over all simplices $T \subset \R^n$ equals $n$, or there exists some constant $c(n,j) < n$ such that any convex body $K \subset \R^n$ satisfies $\alpha_j(K) \leq c(n,j)$ (see \cref{chap:main-theorem} for details).
In other words, we can determine whether the inequality $\alpha_j(K) \leq n$ is tight by checking only simplices.

Although this remains a challenging problem in general, we are able to solve it in the planar case for the first Minkowski chirality.

\begin{theorem}\label{thm:triangle}
Let $K\subset\R^2$ be a triangle.
Then the infimum in the definition of $\alpha_1(K)$ is attained at some affine subspace $U$ of $\R^2$ that is necessarily
 \begin{enumerate}[label={(\roman*)},leftmargin=*,align=left,noitemsep]
    \item{parallel to the bisector of one of the largest interior angles of $K$,}
    \item{parallel to the bisector of one of the smallest interior angles of $K$, or}
    \item{perpendicular to one of the longest edges of $K$.}
 \end{enumerate}
Moreover, we have
\begin{equation}\label{eq:triangle-range}
   \setcond{\alpha_1(K)}{K\subset\R^2\text{ is a triangle}}=\left[1,\sqrt{2}\right),
\end{equation}
with $\alpha_1(K)=1$ precisely for isosceles triangles.
\end{theorem}

The above approach yields the upper bound $\alpha_1(K) < 1.95$ for all convex bodies $K \subset \R^2$, see \eqref{eq:c21}.

The question of how large $\alpha_j(K)$ can be for general $n$ and $j$ is still open, as even deciding whether the inequality $\alpha_j(K) \leq n$ is actually tight appears to be difficult.
Instead, we focus on a special class of convex bodies and answer the first question for planar point-symmetric convex bodies: the upper bound from \cref{thm:dbm} becomes $\sqrt{2}$ in this case, and the following two theorems show that this bound is reached precisely by a specific family of parallelograms.
The second theorem uses the \cemph{John ellipsoid} $\John(K)$ of a convex body $K \subset \R^n$, which is the unique volume-maximal ellipsoid contained in $K$ (see \cref{chap:preliminaries} for details).

\begin{theorem}\label{thm:upperboundsymmetric}
Let $K\subset\R^2$ be a point-symmetric convex body with $d_{BM}(K,P) \geq 1 + \eps$ for a parallelogram $P\subset\R^2$ and some $\eps>0$.
Then 
\begin{equation*}
    \alpha_1(K) < \sqrt{2} \left( 1 - \frac{\eps}{10} \right).
\end{equation*}
\end{theorem}

\begin{theorem}\label{thm:parallelogram}
    Let $K\subset\R^2$ be a parallelogram.
    Then the infimum in the definition of $\alpha_1(K)$ is attained at some affine subspace $U$ of $\R^2$ that is necessarily parallel to
 \begin{enumerate}[label={(\roman*)},leftmargin=*,align=left,noitemsep]
    \item{the bisector of an angle formed by consecutive edges of $K$,}
    \item{the bisector of an angle formed by the diagonals of $K$, or}
    \item{a principal axis of the John ellipse $\John(K)$ of $K$.} 
 \end{enumerate}
    Moreover, we have
\begin{equation}\label{eq:parallelogram-range}
   \setcond{\alpha_1(K)}{K\subset\R^2\text{ is a parallelogram}}=\left[1,\sqrt{2}\right],
\end{equation}
with $\alpha_1(K)=1$ precisely for rectangles and rhombuses. Moreover, $\alpha_1(K)=\sqrt{2}$ if and only if the angles between the diagonals coincide with the interior angles and the ratio between the lengths of the longer edges and the shorter edges is at least $\sqrt{2}$.
\end{theorem}

Our paper is organized as follows.
We start with definitions, notations, and preliminaries in \cref{chap:preliminaries}.
We proceed with basic properties of $\alpha_j$ such as similarity invariance and continuity with respect to the Hausdorff distance in \cref{chap:basic-properties}.
Afterwards, we show inequalities that compare certain Minkowski chiralities, which lead to detailed proofs of \cref{thm:dbm,thm:upperboundsymmetric} in \cref{chap:main-theorem}.
Lastly, we turn to the planar case and show \cref{thm:parallelogram,thm:triangle} in \cref{chap:planar,chap:triangle}. \cref{chap:planar,chap:triangle} also contain a detailed analysis of the explicit 1st Minkowski chirality of arbitrary parallelograms and triangles depending on various natural parametrizations. 

 \section{Preliminaries}\label{chap:preliminaries}
\subsection{General notation}
For $x=(x_1,\ldots,x_n)^\top \in \R^n$, the \cemph{Euclidean norm} is given by $\gauge{x}\defeq \sqrt{\sum_{i=1}^nx_i^2}$ and the \cemph{Euclidean unit ball} is denoted by $\B_2^n\defeq\setcond{x\in\R^n}{\gauge{x}\leq 1}$.
Let $X,Y \subset \R^n$, $z\in\R^n$, and $\mu >0$.
The \cemph{Minkowski sum} of $X$ and $Y$ is given by $X + Y \defeq \setcond{x + y}{x \in X, y \in Y}$.
Sets of the form $X + z \defeq z + X \defeq \setn{z} + X$ and $\mu X \defeq \setcond{\mu x}{x\in X}$ are called \cemph{translates} and \cemph{dilates} of $X$, respectively.
We abbreviate $(-1) X$ and $X + ((-1) Y)$ by $-X$ and $X - Y$, respectively.
The set $X$ is \cemph{$z$-symmetric} if $X-z = z-X$,
and \cemph{point-symmetric} if it is $z$-symmetric for some $z \in \R^n$.
We write $X \subset_t Y$ if there exists $z \in \R^n$ such that $X \subset Y + z$.
Similarly, if there exists $z\in\R^n$ such that $X=Y+z$, we shall write $X=_t Y$ for short.
If $f:\R^n\to\R^n$, then $f(X)\defeq \setcond{f(x)}{x\in X}$ denotes the image of $X$ under $f$.
We say that $f$ is a \cemph{similarity transform} if it is a map of the form $f(x) = rAx+b$ with $r > 0$, $b \in \R^n$, and $A \in \R^{n\times n}$ orthogonal.

We write $\conv(X)$, $\aff(X)$, $\bd(X)$, and $\inte(X)$ for the \cemph{convex hull}, \cemph{affine hull}, \cemph{boundary}, and \cemph{interior} of $X$, respectively.
(For the sake of readability, we will omit the parentheses if the set $X$ is written with curly brackets.)
For $x,y \in \R^n$, the closed \cemph{line segment} connecting them is given by $[x,y]$ where we replace the brackets by parentheses if we wish to exclude the respective endpoint from the line segment.
For $a \in \R^n \setminus \setn{0}$ and $\beta \in \R$, the \cemph{hyperplane} $H_{(a,\beta)}$ is given by $\setcond{x \in \R^n}{a^\top x = \beta}$.
We denote by $H_{(a,\beta)}^\leq \defeq \setcond{x \in \R^n}{a^\top x \leq \beta} $ and $H_{(a,\beta)}^\geq \defeq \setcond{x \in \R^n }{ a^\top x \geq \beta}$ the \cemph{halfspaces} bounded by $H_{(a,\beta)}$.
A halfspace $H^\leq \subset \R^n$ \cemph{supports} a set $X \subset \R^n$ at $x \in X$ if $X \subset H^\leq$ and $x \in \bd(H^\leq)$,
and a hyperplane $H$ supports $X$ at $x$ if one of the halfspaces bounded by $H$ does.
We write $h_X(a) \defeq \sup\setcond{ a^\top x }{ x \in X}$ for the \cemph{support function} of $X$.
The \cemph{polar set} of $X$ is defined as $X^\circ \defeq \setcond{ a \in \R^n }{ h_X(a) \leq 1 }$.
We write $\vol(X)$ for the $n$-dimensional Lebesgue measure (volume) of $X$ if $X$ is measurable.

By a \cemph{triangle} and a \cemph{parallelogram}, we always refer to non-degenerate convex polygons, i.e., their vertices are not elements of a single straight line.
The \cemph{bisector} of an interior angle of a triangle is always the one that has non-empty intersection with the interior of the polygon.

\subsection{Convex bodies, radii, and optimal containment}
We denote by $\CK^n$ the family of \cemph{convex bodies} in $\R^n$, i.e., compact convex sets with nonempty interior.
For $K,L \in \CK^n$, their \cemph{Hausdorff distance} is given by $d_H(K,L) \defeq \inf \setcond{ \lambda > 0 }{ K \subset L + \lambda \B_2^n\text{ and } L \subset K + \lambda \B_2^n}$.
By \cite[Lemma~1.8.14]{Schneider2014}, we equivalently have
$d_H(K,L) =\max_{\gauge{u}=1}\abs{h_K(u) - h_L(u)}$.
The point-wise inequality $h_K(x) \leq h_L(x) + \max_{\gauge{u}=1}\abs{h_K(u) - h_L(u)} h_{\B_2^n}(x) = h_{L + d_H(K,L) \B_2^n}(x)$ for all $x\in\R^n$ thus implies $K \subset L + d_H(K,L) \B_2^n$, so the Hausdorff distance is always attained as a minimum.

For $K,C \in \CK^n$,
the \cemph{circumradius} and the \cemph{inradius} of $K$ with respect to $C$ are defined as 
\begin{equation*}R(K,C)\defeq \inf\setcond{\lambda>0}{K\subset_t \lambda C} \quad \text{and} \quad r(K,C)\defeq \sup\setcond{\lambda>0}{\lambda C\subset_t K},
\end{equation*}
respectively.
The $j$th Minkowski chirality can be written with the help of the circumradius as 
\begin{equation}\label{eq:chirality}
   \alpha_j(K)\defeq \inf\setcond{R(K,\Phi_U(K))}{U\subset\R^n\text{ affine subspace, }\dim(U)=j}. \end{equation}

If $K=-K$ and $C=-C$, 
translations can be omitted in the definition of circum- and inradius, i.e., 
\begin{equation*}
R(K,C)=\inf\setcond{\lambda>0}{K\subset \lambda C} \quad \text{and} \quad r(K,C)=\inf\setcond{\lambda>0}{\lambda C\subset K}.
\end{equation*}

The following result, taken from \cite[Theorem~2.3]{BrandenbergKonig2013}, helps with checking whether a containment $K\subset C$ is \cemph{optimal}, meaning that additionally $R(K,C)=1$.

\begin{proposition}\label{thm:brandenberg-koenig}
Let $K,C \in \CK^n$. Then $K$ is optimally contained in $C$
if and only if \begin{enumerate}[label={(\roman*)},leftmargin=*,align=left,noitemsep]
\item{$K\subset C$, and}
\item{there exist some $a^1,\ldots,a^m\in\R^n\setminus\setn{0}$ such that $h_K(a^i) = h_C(a^i)$ for $i\in\setn{1,\ldots,m}$ and $0\in\conv\setn{a^1,\ldots,a^m}$.}
\end{enumerate}
\end{proposition}

By \cref{thm:brandenberg-koenig}, $K \subset C$ is optimal for triangles if and only if every edge of $C$ contains a vertex of $K$, and for parallelograms if and only if $C$ has two opposite edges each containing a vertex of $K$.

As in the introduction, the \cemph{John ellipsoid} $\John(K)$ of a convex body $K \subset \R^n$ is the unique volume-maximal ellipsoid contained in $K$, see \cite[Theorem~11.1]{Gruber2007}.
The John ellipsoid is affine equivariant,  meaning that for any convex body $K \in \CK^n$, any invertible linear transformation $A \in \gl(\R^n)$, and any translation vector $b \in \R^n$, we have $ \John(A(K) + b) = A(\John(K)) + b$, see \cite[Section~8.4.3]{BoydVandenberghe2004}.
By John's theorem \cite[Theorem~11.2]{Gruber2007}, the Euclidean ball $\B_2^n$ is the John ellipsoid of the cube $[-1,1]^n$.
Since affine transformations preserve midpoints and parallelism of hyperplanes, the John ellipsoid of a parallelotope is the ellipsoid that is tangent to all facets of the parallelotope at their midpoints.

\subsection{Grassmanian, reflections}

For $j\in\setn{0,\ldots,n}$, we denote by $\gr{j}{\R^n}$ the \cemph{Grassmannian}, i.e., the set of all $j$-dimensional linear subspaces of $\R^n$.
This set is topologized as a quotient space of the Stiefel manifold and as such it is homeomorphic to the set
\begin{equation*}
\setcond{M\in\R^{n\times n}}{M=M^\top =M^2,\trace(M)=j}
\end{equation*}
of trace-$j$ symmetric idempotent real $n\times n$-matrices equipped with the subspace topology of $\R^{n\times n}$, see \cite[\textsection~5]{MilnorSt1974}.
These matrices are precisely the \cemph{orthoprojectors} $P_U:\R^n\to\R^n$ onto $j$-dimensional subspaces $U$ of $\R^n$, where $P_U(x)\in \R^n$ is for $x \in \R^n$ uniquely determined by the conditions $P_U(x)\in U$ and $\gauge{P_U(x) - x}\leq\gauge{z-x}$ for all $z\in U$.
Since the map $U\mapsto P_U$ is bijective and the set of trace-$j$ symmetric idempotent real $n\times n$-matrices is a closed subset of the unit sphere of the spectral norm $\gauge{\cdot}$ on $\R^{n\times n}$, we see that $d_G(U,V)\defeq \gauge{P_U-P_V}$ defines a metric on $\gr{j}{\R^n}$, and that $(\gr{j}{\R^n},d_G)$ is a compact metric space.

For $U \subset \R^n$ an affine subspace, the reflection across $U$ is given by $\Phi_U: \R^n \to \R^n$, $\Phi_U(x)=2P_U(x)-x$.
Note that for $x\in\R^n$, we have
\begin{equation}\label{eq:reflection-calculus}
\Phi_U(x)=\Phi_{U-U}(x)+2P_U(0).
\end{equation}

\section{Basic properties}\label{chap:basic-properties}

This section is devoted to verifying some basic properties of the Minkowski chiralities that Grünbaum \cite{Grunbaum1963} demanded of any (a)symmetry measure.
In particular, we show their similarity invariance and Hausdorff continuity.
We also obtain results on ratios of certain Minkowski chiralities, which prepares us for the proofs of the upper bounds on the chiralities in the next section.

\subsection{Similarity invariance}
We begin this subsection with a slight simplification of the definition of the Minkowski chirality.
The translation invariance of the circumradius together with \eqref{eq:reflection-calculus} shows for $K \in \CK^n$ and an affine subspace $U$ of $\R^n$ of dimension $j\in\setn{0,\ldots,n}$ that $R(K,\Phi_U(K)) = R(K,\Phi_{U-U}(K))$.
Taking the infimum over the $j$-dimensional affine subspaces $U$ in \eqref{eq:chirality}, we get 
\begin{equation}\label{eq:chirality-linear}
    \alpha_j(K)= \inf\setcond{R(K,\Phi_U(K))}{U\in\gr{j}{\R^n}}.
\end{equation}

Next, we show that the Minkowski chiralities are similarity invariant. 

\begin{lemma}\label{thm:similarity-invariance}
Let $K \in\CK^n$ and $j \in \setn{0, \ldots, n}$.
Then $\alpha_j(f(K))=\alpha_j(K)$ for any similarity transform $f:\R^n\to\R^n$. \end{lemma}
\begin{proof}
The general case is a composition of the cases of translations and linear similarity transforms.

First, assume that $f$ is a translation, i.e., there exists $b\in\R^n$ such that $f(x)=x+b$ for all $x\in\R^n$.
Using the linearity of $\Phi_U$ for $U \in \gr{j}{\R^n}$ and the translation invariance of the circumradius, 
\begin{equation*}
    R(K+b,\Phi_U(K+b))
    = R(K+b,\Phi_U(K)+\Phi_U(b))
    = R(K,\Phi_U(K)).
\end{equation*}

Now take the infimum over $U\in\gr{j}{\R^n}$ to obtain $\alpha_j(K+b)=\alpha_j(K)$ from \eqref{eq:chirality-linear}.

Second, assume that the similarity transform $f:\R^n\to\R^n$ is linear.
Since $f$ is invertible, the containment $f(K)\subset x+\lambda C$ for $x \in \R^n$ and $\lambda > 0$ is equivalent to $K\subset f^{-1}(x)+\lambda f^{-1}(C)$.
It follows that $R(f(K),C)=R(K,f^{-1}(C))$.
Now, for $U \in \gr{j}{\R^n}$, consider
\begin{equation*}
    R(f(K),\Phi_U(f(K)))=R(K,f^{-1}(\Phi_U(f(K))))=R(K,\Phi_{f^{-1}(U)}(K))
\end{equation*}
and take the infimum over $U\in\gr{j}{\R^n}$ to obtain $\alpha_j(f(K))=\alpha_j(K)$ from \eqref{eq:chirality-linear}.
Note that when $U$ traverses $\gr{j}{\R^n}$, then so does $f^{-1}(U)$.
\end{proof}
Let us remark that the Minkowski asymmetry $\alpha_0$ enjoys an even stronger invariance property, namely the invariance under invertible affine transformations, i.e., $\alpha_0(A(K)+b)=\alpha_0(K)$ 
for all $K\in\CK^n$, $b\in\R^n$, and $A\in \gl(\R^n)$.
This is evident from the equivalence of $K\subset_t \lambda (-K)$ and $A(K)\subset_t \lambda (-A(K))$ for $\lambda >0$.
In contrast, the $j$th Minkowski chirality $\alpha_j$ is not affinely invariant when $j\neq 0$.
For instance, a rectangle $K$ satisfies $\alpha_1(K)=1$, but there exist affine images $A(K)$ with $\alpha_j(A(K))>1$, namely when $A(K)$ is a parallelogram that is neither a rectangle nor a rhombus.

\subsection{Existence of optimal subspaces and continuity}
In this subsection, we prove that $\alpha_j$ is Hausdorff continuous and that the infimum in \eqref{eq:chirality-linear} is attained by some subspace $U^\ast\in\gr{j}{\R^n}$.

We start with some preparatory results regarding the continuity of the circumradius $R:\CK^n\times \CK^n\to \R$ and the map $\Phi:\CK^n\times \gr{j}{\R^n}\to \CK^n$, $(K,U)\mapsto\Phi_U(K)$.
These maps are defined on the Cartesian products $\CK^n\times\CK^n$ and $\CK^n\times \gr{j}{\R^n}$, on which we consider the metrics $d_H+d_H$ and $d_H+d_G$ given by
$(d_H+d_H)((K^1,C^1),(K^2,C^2))\defeq d_H(K^1,K^2)+d_H(C^1,C^2)$ and $(d_H+d_G)((K^1,U^1),(K^2,U^2))\defeq d_H(K^1,K^2)+d_G(U^1,U^2)$ for $U^1,U^2\in\gr{j}{\R^n}$ and $K^1,K^2,C^1,C^2\in\CK^n$.

\begin{lemma}\label{thm:circumradius-continuous}
    The map $R:\CK^n\times \CK^n\to \R$ is continuous.
\end{lemma}

The proof is based on \cite[Proposition~1.2.1]{Gonzalez2013} about the Hausdorff continuity of $R(\cdot,\B_2^n)$ on $\CK^n$.

\begin{proof}
Let $((K^i,C^i))_{i\in\N}$ be a convergent sequence in the metric space $(\CK^n\times \CK^n,d_H+d_H)$ with limit $(K,C)\in\CK^n\times\CK^n$.
Then $(K^i)_{i\in\N}$, $(C^i)_{i\in\N}$ are convergent sequences in $(\CK^n,d_H)$ with limits $K$ and $C$, respectively.
By the translation invariance of the circumradius, we may assume that $0\in\inte(K)$ and $0\in\inte(C)$, i.e., there exists $r>0$ such that $r\B_2^n \subset K$ and $r\B_2^n \subset C$.
Fix $i\in\N$.
Then 
\begin{align*}
    K^i
    & \subset K+d_H(K^i,K)\B_2^n
    \subset K+\frac{d_H(K^i,K)}{r}K
    = \left( 1+\frac{d_H(K^i,K)}{r} \right)K
    \\
\intertext{and, analogously,}
    C^i
    & \subset C+(d_H(C^i,C))\B_2^n
    \subset C+\frac{d_H(C^i,C)}{r}C
    = \left( 1+\frac{d_H(C^i,C)}{r} \right)C.
\end{align*}
For $i$ sufficiently large, we have $d_H(K^i,K)<r$ and $d_H(C^i,C)<r$, so
\begin{align*}
    \left( 1-\frac{d_H(K^i,K)}{r} \right)K + \frac{d_H(K^i,K)}{r}K
    = K
    \subset K^i+d_H(K^i,K)\B_2^n
    & \subset K^i+\frac{d_H(K^i,K)}{r}K
    \\
\intertext{and}
    \left( 1-\frac{d_H(C^i,C)}{r} \right)C + \frac{d_H(C^i,C)}{r}C
    = C
    \subset C^i+d_H(C^i,C)\B_2^n
    & \subset C^i+\frac{d_H(C^i,C)}{r}C.
\end{align*}
Together with the cancellation property, cf.\ \cite[p.~48]{Schneider2014}, we conclude
\begin{equation*}
    \left( 1-\frac{d_H(K^i,K)}{r} \right)K
    \subset K^i
        \quad \text{and} \quad
    \left( 1-\frac{d_H(C^i,C)}{r} \right)C
    \subset C^i.
\end{equation*}
Thus, the monotonicity of the circumradius under set inclusions gives
\begin{equation*}
  \frac{\left( 1-\frac{d_H(K^i,K)}{r} \right)}{\left( 1+\frac{d_H(C^i,C)}{r} \right)}R(K,C)
  \leq R(K^i,C^i)
  \leq \frac{\left( 1+\frac{d_H(K^i,K)}{r} \right)}{\left( 1-\frac{d_H(C^i,C)}{r} \right)} R(K,C).
\end{equation*}
This shows that $\lim_{i\to\infty} R(K^i,C^i)=R(K,C)$.
\end{proof}

Next, we show that the reflection of convex bodies across subspaces is jointly continuous with respect to the two input variables.

\begin{lemma}\label{lem:reflection-continuous}
Let $j\in\setn{0,\ldots,n}$.
The map $\Phi:\CK^n\times \gr{j}{\R^n}\to \CK^n$, $(K,U)\mapsto\Phi_U(K)$ is continuous.
\end{lemma}
\begin{proof}

We first show for $K \in \CK^n$ that the map $\Phi_\bullet(K):\gr{j}{\R^n}\to \CK^n$, $U\mapsto\Phi_U(K)$ is continuous.

Let $(U^i)_{i\in\N}$ be a convergent sequence in the metric space $(\gr{j}{\R^n},d_G)$ with limit $U$.
We show 
\begin{equation*}
    \lim_{i\to\infty}d_H(\Phi_{U^i}(K),\Phi_U(K))=0
\end{equation*}
using \cite[Theorem~1.8.8]{Schneider2014}.

First, let $x\in\Phi_U(K)$.
Then there exists $y\in K$ such that $x=\Phi_U(y)$.
For $x^i=\Phi_{U^i}(y)\in \Phi_{U^i}(K)$, 
\begin{equation*}
    \gauge{x^i-x}=\gauge{\Phi_{U^i}(y)-\Phi_U(y)}\leq \gauge{\Phi_{U^i}-\Phi_U}\gauge{y}=2d_G(U^i,U)\gauge{y}\xrightarrow{i\to \infty}0, \end{equation*}
i.e., $x=\lim_{i\to\infty}x^i$.
Second, let $(i_j)_{j\in\N}$ be an increasing sequence in $\N$, and suppose that $x^{i_j}\in \Phi_{U^{i_j}}(K)$ for all $j\in\N$.
Further suppose that $(x^{i_j})_{j\in\N}$ is a convergent sequence whose limit shall be denoted by $x^\ast\in\R^n$.
For all $j\in\N$, there exists $y^{i_j}\in K$ such that $x^{i_j}=\Phi_{U^{i_j}}(y^{i_j})$.
We have
\begin{align*}
    \gauge{y^{i_j}-\Phi_U(x^\ast)}&=\gauge{\Phi_{U^{i_j}}(x^{i_j})-\Phi_U(x^\ast)}\\
    &=\gauge{\Phi_{U^{i_j}}(x^{i_j})-\Phi_{U^{i_j}}(x^\ast)+\Phi_{U^{i_j}}(x^\ast)-\Phi_U(x^\ast)}\\
    &\leq \gauge{\Phi_{U^{i_j}}}\gauge{x^{i_j}-x^\ast}+\gauge{\Phi_{U^{i_j}}-\Phi_U}\gauge{x^\ast}\\
    &= \gauge{x^{i_j}-x^\ast}+2d_G(U^{i_j},U)\gauge{x^\ast}\xrightarrow{j\to \infty}0.
\end{align*}
Thus, the sequence $(y^{i_j})_{j\in\N}$ is a convergent sequence with limit $\Phi_U(x^\ast)$.
Since this is a sequence in the closed set $K$, the limit is an element of $K$, too.
But $\Phi_U(x^\ast)\in K$ is equivalent to $x^\ast\in \Phi_U(K)$.
This completes the proof of the continuity of $\Phi_\bullet(K)$.

Now, let $((K^i,U^i))_{i\in\N}$ be a convergent sequence in the metric space $(\CK^n\times \gr{j}{\R^n},d_H+d_G)$ with limit $(K,U)\in\CK^n\times \gr{j}{\R^n}$.
Then $(K^i)_{i\in\N}$ is a convergent sequence in $(\CK^n,d_H)$ with limit $K$, and $(U^i)_{i\in\N}$ is a convergent sequence in $(\gr{j}{\R^n},d_G)$ with limit $U$.
We see
\begin{align*}
    d_H(\Phi_{U^i}(K^i),\Phi_U(K))&\leq d_H(\Phi_{U^i}(K^i),\Phi_{U^i}(K))+d_H(\Phi_{U^i}(K),\Phi_U(K))\\
    &=d_H(K^i,K)+d_H(\Phi_{U^i}(K),\Phi_U(K))\xrightarrow{i\to \infty}0,
\end{align*}
i.e., we have $\lim_{i\to\infty}\Phi_{U^i}(K^i)=\Phi_U(K)$ in $(\CK^n,d_H)$.
\end{proof}

Lastly, we establish the continuity of $\alpha_j$ and show that the infimum in \eqref{eq:chirality-linear} is attained.
In particular, we verify that any convex body $K$ contains at least one point ($x^\ast$ below) that acts for $\alpha_j$ like a \emph{Minkowski center} would for $\alpha_0$ (see \cite[Definition~$3.1$]{BrandenbergKonig2015}).
This is important for some constructions in the next subsection.

\begin{proposition} \label{prop:chirality-attained}
    Let $K\in\CK^n$ and $j\in\setn{0,\ldots,n}$.
    There exist $U^\ast \in \gr{j}{\R^n}$ and $x^\ast \in K$ such that $K-x^\ast \subset \alpha_j(K) \Phi_{U^\ast}(K-x^\ast)$.
    Moreover, the map $\alpha_j:\CK^n\to\R$ is Hausdorff continuous.
\end{proposition}
\begin{proof}
By \cite[Lemma 2.2]{BrandenbergKonig2015}, we obtain $R(K,\Phi_U(K))<\infty$ for all $U\in\gr{j}{\R^n}$, so $\alpha_j(K)<\infty$.
By \cref{lem:reflection-continuous,thm:circumradius-continuous}, the map $\gr{j}{\R^n}\to\R$, $U\mapsto R(K,\Phi_U(K))$ is a continuous map over a compact domain.
The existence of $U^\ast \in \gr{j}{\R^n}$ with $R(K,\Phi_{U^\ast}(K)) = \alpha_j(K)$ is now a consequence of Weierstrass's theorem, and the continuity of $\alpha_j:\CK^n\to\R$ follows from \cref{lem:reflection-continuous} and the compactness of the metric space $(\gr{j}{\R^n},d_G)$.

Again by \cite[Lemma 2.2]{BrandenbergKonig2015}, there exists some $x \in \R^n$ with $K \subset x + \alpha_j(K) \Phi_{U^\ast}(K)$.
There also exist $y \in U^\ast$ and $z \in (U^\ast)^\perp$ such that $x = y+z$.
Next, we want to choose $y' \in U^\ast$ and $z' \in (U^\ast)^\perp$ with $(1-\alpha_j(K)) y' = y$ and $(1+\alpha_j(K)) z' = z$.
There are clearly no obstructions in doing so when $\alpha_j(K)>1$.
If, on the other hand, we have $\alpha_j(K) = 1$, then
\begin{equation*}
    \Phi_{U^\ast}(K)
    \subset \Phi_{U^\ast}(y + z + \Phi_{U^\ast}(K))
    = y - z + K
\end{equation*}
shows that $K \subset y + z + \Phi_{U^\ast}(K) \subset 2 y + K$.
From the cancellation property, cf.\ \cite[p.~48]{Schneider2014}, we conclude $y=0$.
In other words, also in the case when $\alpha_j(K)=1$, the choice of $y'\in U^\ast$ with $(1-\alpha_j(K)) y' = y$ is possible (and, moreover, any $y'\in U^\ast$ is admissible).
Define $x' \defeq y' + z'$.
Then
\begin{align}
    \alpha_j(K) \Phi_{U^\ast}(K - x')
    & = \alpha_j(K) \Phi_{U^\ast}(K) - \alpha_j(K) y' + \alpha_j(K) z'\nonumber
    \\
    & \supset K - y - z + (1-\alpha_j(K)) y' + (1+\alpha_j(K)) z' - y' - z'\label{eq:existence}
    = K - x'.
\end{align}

It remains to show that $x'$ can be chosen from $K$.
To this end, we first claim that $\inte(K-x') \cap U^\ast \neq \emptyset$.
Towards a contradiction, let us assume that $\inte(K-x') \cap U^\ast =\emptyset$.
Then \cite[Theorem~1.3.8]{Schneider2014} yields the existence of $a \in (U^\ast)^\perp \setminus \setn{0}$ with $h_{K-x'}(a) \leq 0$ such that there exists $v \in K-x'$ with $a^\top v < 0$.
However, we then have
\begin{equation*}
    h_{\alpha_j(K) \Phi_{U^\ast}(K-x')}(-a)
    = \alpha_j(K) h_{K-x'}(\Phi_{U^\ast}(-a))
    = \alpha_j(K) h_{K-x'}(a)
    \leq 0,
\end{equation*}
which contradicts $v \in K-x' \subset \alpha_j(K) \Phi_{U^\ast}(K-x')$.
Therefore, there must exist $u' \in \inte(K-x') \cap U^\ast$.
If now $\alpha_j(K) = 1$, then $x^\ast \defeq x' + u'$ is an element of $K$ and, from the discussion for \eqref{eq:existence}, we see that $x^\ast=(y'+u')+z'\in U^\ast+(U^\ast)^\perp$ satisfies $K-x^\ast\subset \alpha_0(K)\Phi_{U^\ast}(K-x^\ast)$.
If instead $\alpha_j(K) > 1$, then we argue that $x^\ast \defeq x'\in K$.
Towards a contradiction, let us assume that $x'\notin K$.
Then we have $0 \notin K' \defeq (K-x') \cap U^\ast$.
From \cite[Theorem~1.3.4]{Schneider2014}, we conclude that there exists $a \in U^\ast$ such that $h_{K'}(a) < 0$.
Since $K' \subset (\alpha_j(K) \Phi_{U^\ast}(K-x')) \cap U^\ast = \alpha_j(K) K'$, we obtain
\begin{equation*}
    h_{K'}(a)
    \leq h_{\alpha_j(K) K'}(a)
    = \alpha_j(K) h_{K'}(a).
\end{equation*}
However, this contradicts $h_{K'}(a) < 0$ and $\alpha_j(K) > 1$.
\end{proof}

\subsection{Relating chiralities to each other}

The next theorem relates the $j$th Minkowski chiralities of different convex bodies to each other.
It generalizes \cite[Theorem~6.1]{BrandenbergKonig2015}, where the same inequality is shown, but restricted to the Minkowski asymmetry $\alpha_0$.
The ratio between circumradius and inradius, which appears on the right-hand side of our estimate, has been studied in \cite{Toth2013} as a distance measure in the context of the Banach--Mazur distance and the Minkowski asymmetry.
This connection allows us to obtain the general estimates on the Minkowski chiralities in the next section, extending the importance of the following theorem beyond the context of comparing Minkowski chiralities.

\begin{theorem} \label{thm:fejestoth}
Let $K,L \in \CK^n$ and $j \in  \setn{0,\ldots,n}$.
Then
\begin{equation}\label{eq:ft}
    \max \setn{ \frac{\alpha_j(K)}{\alpha_j(L)}, \frac{\alpha_j(L)}{\alpha_j(K)} }   \leq \frac{R(K,L)}{r(K,L)}=\frac{R(L,K)}{r(L,K)}.
\end{equation}
For every $K\in \CK^n$ and $\beta \in [1,\alpha_j(K)]$, there exists $L\in\CK^n$ with $\alpha_j(L)=\beta$ and
equality in \eqref{eq:ft}. \end{theorem}
\begin{proof}
The identity $\frac{R(K,L)}{r(K,L)}=\frac{R(L,K)}{r(L,K)}$ is direct from the definitions of $R$ and $r$.
Let $U \in \gr{j}{\R^n}$ with $R(L,\Phi_U(L)) = \alpha_j(L)$ be obtained from \cref{prop:chirality-attained}.
Then $L \subset_t \alpha_j(L) \Phi_U(L)$, and, hence,
\begin{align*}
    K
    & \subset_t R(K,L) L
    \subset_t R(K,L) \alpha_j(L) \Phi_U(L)
    \\
    & \subset_t R(K,L) \alpha_j(L) \Phi_U(r(K,L)^{-1} K)
    = \alpha_j(L) \frac{R(K,L)}{r(K,L)} \Phi_U(K).
\end{align*}
Consequently, we have $R(K,\Phi_U(K))\leq \alpha_j(L) \frac{R(K,L)}{r(K,L)}$.
Now, use \eqref{eq:chirality-linear}
to conclude
\begin{equation*}\alpha_j(K)  \leq \alpha_j(L)\frac{R(K,L)}{r(K,L)}.
\end{equation*}
Reversing the roles of $K$ and $L$ gives 
\begin{equation*}\alpha_j(L)   \leq \alpha_j(K) \frac{R(L,K)}{r(L,K)} = \alpha_j(K)\frac{R(K,L)}{r(K,L)}.
\end{equation*}
Combining those two results yields the claimed inequality.

Let us now turn to the equality case.
According to \cref{prop:chirality-attained}, there exist $U \in \gr{j}{\R^n}$ and $x \in K$ with $K-x \subset \alpha_j(K) \Phi_U(K-x)$.
We define for $\beta \in [1,\alpha_j(K)]$ the convex body
\begin{equation*}
    L \defeq \conv((\beta (K-x)) \cup \Phi_U(K-x)).
\end{equation*}
Then $r(L,K) \geq \beta$ since $\beta (K-x) \subset L$.
Moreover, $0 \in K-x$ and $\alpha_j(K) \geq \beta$ show
\begin{equation*}
    \beta (K-x)
    \subset \alpha_j(K) (K-x),
\end{equation*}
so $L \subset \alpha_j(K) (K-x)$.
We conclude $R(L,K) \leq \alpha_j(K) $ and $\frac{R(K,L)}{r(K,L)}=\frac{R(L,K)}{r(L,K)} \leq \frac{\alpha_j(K) }{\beta}$.
The already proven inequality \eqref{eq:ft} implies $\alpha_j(L) \geq \beta$.
It remains to show $\alpha_j(L) \leq \beta$.
This is immediate from
\begin{equation*}
    \Phi_U(L)
    = \conv( (\beta \Phi_U(K-x)) \cup (K-x) )
    \subset \conv( (\beta^2 (K-x)) \cup (\beta \Phi_U(K-x)) )
    = \beta L,
\end{equation*}
where we used $\beta \geq 1$ and $0 \in K-x$.
\end{proof}

\begin{remark}\label{rem:fejestoth}
It follows from \cref{thm:fejestoth} that $\alpha_j(K)\leq \frac{R(K,L)}{r(K,L)}$ for all $j\in\setn{0,\ldots,n}$ and $K,L\in\CK^n$ with $\alpha_j(L)=1$.
Moreover, for every $j\in\setn{0,\ldots,n}$ and $K\in \CK^n$, there exists $L\in\CK^n$ with $\alpha_j(L)=1$ such that $\alpha_j(K)=\frac{R(K,L)}{r(K,L)}$.
Writing $d_D(K,L) \defeq \frac{R(K,L)}{r(K,L)}$, this shows
\begin{equation}\label{eq:fejestoth}
\alpha_j(K)=\min\setcond{d_D(K,L)}{L\in\CK^n,\,\alpha_j(L)=1}.
\end{equation}
The quantity $d_D$ has been studied as a close relative of the Banach--Mazur distance \cite{Toth2013}.
This gives another interpretation of $\alpha_j(K)$ as quantifying the distance from the family of convex bodies that are symmetric with respect to reflection at an appropriate $j$-dimensional subspace.
Since $\alpha_0$ is affinely invariant, \eqref{eq:fejestoth} generalizes the identity $\alpha_0(K)=\min\setcond{d_{BM}(K,L)}{L\in\CK^n,\,\alpha_0(L)=1}$ from \cite[Proposition~3.1]{BrandenbergGonzalez2017}.
\end{remark}

While \cref{thm:fejestoth} relates the values of the same Minkowski chirality of different convex bodies to each other, we can also compare different Minkowski chiralities associated with a single convex body.
This requires the following simple observations.

\begin{proposition}\label{thm:orthogonal-axis-central-symmetry}
Let $K \in \CK^n$ and $U \subset \R^n$ be a linear subspace.
Then
\begin{enumerate}[label={(\roman*)},leftmargin=*,align=left,noitemsep]
\item{$R(K,\Phi_U(K))=R(K,\Phi_{U^\perp}(-K))$, and \label{central-symmetry-perp}}
\item{$R(K,\Phi_U(K)) = R(K^\circ,\Phi_U(K^\circ))$ if $K$ is $0$-symmetric.\label{central-symmetry-circ}}
\end{enumerate}
\end{proposition}
\begin{proof}
For \ref{central-symmetry-perp}, recall that $\Phi_U(\Phi_U(x)) = x$ and $\Phi_U(\Phi_{U^\perp}(x))=-x$ for all $x \in \R^n$.
Thus,
\begin{equation*}
    \Phi_U(K)
    = \Phi_U(\Phi_U(\Phi_{U^\perp}(-K)))
    = \Phi_{U^\perp}(-K).
\end{equation*}
For \ref{central-symmetry-circ}, let $A\in\R^{n\times n}$ be the matrix representation of $\Phi_U$ with respect to the standard basis of $\R^n$.
Then the $0$-symmetry of $K$, $\Phi_U(K)$, $K^\circ$, and $\Phi_U(K^\circ)$ shows for $\lambda > 0$ that
\begin{align*}
\lambda \geq R(K,\Phi_U(K))&\Longleftrightarrow K\subset \lambda A(K)\Longleftrightarrow K^\circ \supset \frac{1}{\lambda}A^{-\top}(K^\circ)
\Longleftrightarrow \lambda A^\top (K^\circ)\supset K^\circ \\
&\Longleftrightarrow \lambda \Phi_U (K^\circ)\supset K^\circ \Longleftrightarrow \lambda \geq R(K^\circ,\Phi_U(K^\circ)).\qedhere
\end{align*}
\end{proof}

We are now ready to verify the final result of this section, which highlights the special role of point-symmetry and the Minkowski asymmetry $\alpha_0$ in the context of the general Minkowski chiralities.

\begin{theorem}\label{thm:alpha-symmetry}
Let $K \in \CK^n$ and $j \in \setn{0,\ldots,n}$.
Then
\begin{equation}\label{eq:alpha-symmetry}
    \max \setn{ \frac{\alpha_j(K)}{\alpha_{n-j}(K)}, \frac{\alpha_{n-j}(K)}{\alpha_j(K)} }    \leq \alpha_0(K).
\end{equation}
If $K$ is point-symmetric, then $\alpha_j(K)=\alpha_{n-j}(K)$.
If $K$ is $0$-symmetric, then $\alpha_j(K) = \alpha_j(K^\circ)$.
\end{theorem}
\begin{proof}
Let $U \in \gr{j}{\R^n}$ with $R(K,\Phi_{U^\perp}(K)) = \alpha_{n-j}(K)$ be obtained from \cref{prop:chirality-attained}.
Since $-K\subset_t \alpha_0(K)K$, the monotonicity and translation invariance of the circumradius, together with \cref{thm:orthogonal-axis-central-symmetry}, show
\begin{equation*}
    \alpha_j(K)
    \leq R(K,\Phi_U(K))
    = R(K,\Phi_{U^\perp}(-K))
    \leq R(K,\Phi_{U^\perp}(\alpha_0(K)K))
    = \alpha_0(K) \alpha_{n-j}(K).
\end{equation*}
Reversing the roles of $j$ and $n-j$ proves \eqref{eq:alpha-symmetry}.

Since a point-symmetric convex body $K$ satisfies $\alpha_0(K) = 1$, \eqref{eq:alpha-symmetry} shows $\alpha_j(K) \leq \alpha_{n-j}(K) \leq \alpha_{n-(n-j)}(K)  = \alpha_j(K)$ and consequently $\alpha_j(K) = \alpha_{n-j}(K)$ in this case.

Lastly, if $K$ is $0$-symmetric, \cref{thm:orthogonal-axis-central-symmetry} and \eqref{eq:chirality-linear} show 
\begin{equation*}
    \alpha_j(K^\circ)
    = \inf \setcond{R(K^\circ,\Phi_U(K^\circ))}{U \in \gr{j}{\R^n}}
    = \inf \setcond{R(K,\Phi_U(K))}{U \in \gr{j}{\R^n}}
    = \alpha_j(K).\qedhere
\end{equation*}
\end{proof}

Let us point out that the Minkowski asymmetry $\alpha_0$ on the right-hand side in \eqref{eq:alpha-symmetry} cannot be replaced with the other Minkowski chiralities in general.
For instance, any triangle $T \subset \R^2$ satisfies $\frac{\alpha_0(T)}{\alpha_2(T)} = 2$, yet \cref{thm:triangle} implies $\alpha_1(T) < \sqrt{2}$.
We also have $\alpha_0(T) \neq \alpha_2(T)$, so the point-symmetry in the second part of the theorem cannot be omitted.
Lastly, it is easy to see that $\alpha_j(K) = \alpha_j(K^\circ)$ may fail for general $K \in \CK^n$ with $0 \in \inte(K)$, even if $K$ is point-symmetric (with its center outside the origin).
For example, the square $K \subset \R^2$ with vertices $(2,1)$, $(-2,1)$, $(-2,-3)$, and $(2,-3)$ satisfies $\alpha_0(K) = 1$, but $K^\circ$ is a non-point-symmetric kite with $\alpha_0(K^\circ) > 1$.

\section{General upper bounds on the Minkowski chirality}\label{chap:main-theorem}

We verify \cref{thm:dbm,thm:upperboundsymmetric} in this section.
Their proofs require the following auxiliary result, which states that any ellipsoid is invariant under reflection at a subspace spanned by some of its principal axes.

\begin{proposition}\label{prop:ellipsoid_invar_reflection}
Let $v^1, \ldots, v^n \in \R^n$ form an orthonormal basis and let $\alpha_1, \ldots, \alpha_n > 0$.
Define the ellipsoid $E = \setcond{ x \in \R^n }{ \sum_{i =1}^n \frac{(x^\top v^i)^2}{\alpha_i^2} \leq 1 }$
and the subspace $U \defeq\setcond{\sum_{i=1}^j \mu_i v^i}{\mu_1,\ldots,\mu_j\in\R}$ for some $j \in \setn{0,\ldots,n}$.
Then $\Phi_U(E)= E$ and $\alpha_j(E) = 1$.
\end{proposition} 
\begin{proof}
Since $\Phi_U$ is invertible with $\Phi_U = \Phi_U^{-1}$,
we have $\Phi_U(E) = \Phi^{-1}_U(E)= \setcond{ y \in \R^n }{ \Phi_U(y) \in E }.$
It is therefore enough to show for any $y \in \R^n$ and $i \in \setn{1,\ldots,n}$ that
$(\Phi_U(y)^\top v^i)^2 = (y^\top v^i)^2$.

To this end, recall that $\Phi_U(y)=2 P_U(y) - y$,
where $P_U(y) = \sum_{i=1}^j (y^\top v^i) v^i$ is the orthogonal projection of $y$ onto $U$.
Now, if $i \in \setn{1,\ldots,j}$, we obtain
\begin{equation*}
    \Phi_U(y)^\top v^i
    = 2 \sum_{\ell=1}^j (y^\top v^\ell) ((v^\ell)^\top v^i) - y^\top v^i
    = 2 (y^\top v^i) - y^\top v^i
    = y^\top v^i.
\end{equation*}
If instead $i \in \setn{j+1,\ldots,n}$, we have
\begin{equation*}
    \Phi_U(y)^\top v^i
    = 2 \sum_{\ell =1}^j (y^\top v^\ell) ((v^\ell)^\top v^i) - y^\top v^i
    = 0 - y^\top v^i
    = - y^\top v^i.\qedhere
\end{equation*}
\end{proof}

The following result establishes a direct link between the $j$th Minkowski chirality and the Banach--Mazur distance to the Euclidean ball.
It serves as our main tool for the proofs of \cref{thm:dbm,thm:upperboundsymmetric}.
The inequality below can also be read as a lower bound to the Banach-Mazur distance to the Euclidean ball.
Lower bounds to Banach-Mazur distances are typically difficult to verify, which gives the inequality some additional value as another tool to obtain such lower bounds.
\begin{theorem} \label{thm:dbm_asymm} 
Let $K \in \CK^n$ and $j \in \setn{0,\ldots,n}$.
Then
\begin{equation*}
\alpha_j(K) \leq d_{BM}(K,\B_2^n).
\end{equation*}
Moreover, for every $\beta\in [1,n]$, there exists $K\in\CK^n$ such that $\alpha_0(K)=d_{BM}(K,\B_2^n)=\beta$.
\end{theorem}

\begin{proof}
From \eqref{eq:fejestoth} and \cref{prop:ellipsoid_invar_reflection}, we obtain
\begin{align*}
    \alpha_j(K)
    & = \min \setcond{d_D(K,L)}{L \in \CK^n, \alpha_j(L) = 1}
     \leq \inf \setcond{d_D(K,E)}{E \in \CK^n \text{ ellipsoid}}
    = d_{BM}(K,\B_2^n).
\end{align*}

Towards the second claim, let $T^n \in \CK^n$ be a regular simplex with $\frac{1}{n} \B_2^n \subset T^n \subset \B_2^n$.
For $K \defeq \conv( \B_2^n \cup (\beta T^n) )$, $\beta \in [1,n]$, we have 
\begin{equation*}
    \beta T^n
    \subset K
    \subset n T^n \quad \text{and} \quad \B_2^n
    \subset K
    \subset \beta \B_2^n,
\end{equation*}
which implies $d_{BM}(K,T^n) \leq \frac{n}{\beta}$ and
$\alpha_0(K) \leq d_{BM}(K,\B_2^n) \leq \beta$.
Now, \cref{thm:fejestoth} together with $\alpha_0(T^n) = n$ and the affine invariance of $\alpha_0$ additionally shows
\begin{equation*}
    d_{BM}(K,T^n)
    = \inf \setcond{\frac{R(K,A(T^n))}{r(K,A(T^n))}}{A \in \gl(\R^n)}
    \geq \frac{\alpha_0(T^n)}{\alpha_0(K)}
    \geq \frac{n}{\beta}.
\end{equation*}
Therefore, we must have $\alpha_0(K) = \beta$ and altogether $\beta= \alpha_0(K)\leq d_{BM}(K,\B_2^n)\leq \beta$.
\end{proof}

We now prove the upper bound on the Minkowski chirality for general $n$-dimensional convex bodies.

\begin{proof}[Proof of \cref{thm:dbm}]
The inequality $\alpha_j(K)\geq 1$ becomes clear from a volumetric argument:
If $U\in\gr{j}{\R^n}$, $x\in\R^n$, $\mu>0$, and $K\subset x+\mu\Phi_U(K)$, then $\vol(K)\leq \vol(x+\mu \Phi_U(K))=\mu^n \vol(K)$, so $\mu\geq 1$.
If $\alpha_j(K)=1$, then by \cref{prop:chirality-attained} there exist $U\in \gr{j}{\R^n}$, $x\in K$ such that $K-x\subset \Phi_U(K-x)$, i.e., $K\subset \Phi_{U+x}(K)$.
Since $K$ and $\Phi_{U+x}(K)$ have equal volumes, they coincide.

Towards the upper bound, we obtain from \cite[(10.113)]{Schneider2014} and \cref{thm:dbm_asymm} that 
\begin{equation*}
\alpha_j(K)\leq d_{BM}(K,\B_2^n)\leq n.
\end{equation*}
Moreover, \cite[Corollary~5.3]{BrandenbergGrundbacher2024} and \cite[Corollary~4.3]{BrandenbergKonig2015} together with the affine invariance of $\alpha_0$ imply
\begin{align*}
    \alpha_j(K)
    & \leq d_{BM}(K,\B_2^n)
    = \inf \setcond{\frac{w(A(K),\B_2^n)}{2 r(A(K),\B_2^n)}\frac{2 R(A(K),\B_2^n)}{w(A(K),\B_2^n)} }{A \in \gl(\R^n)}
    \\
    & \leq \frac{\alpha_0(K)+1}{2} \inf \setcond{\frac{2 R(A(K),\B_2^n)}{w(A(K),\B_2^n)}}{A \in \gl(\R^n)}
    \leq \frac{\alpha_0(K)+1}{2}\sqrt{n} ,
\end{align*}
where $w(K,C)\defeq 2r(K-K,C-C)$ is the minimal width of $K \in \CK^n$ with respect to $C \in \CK^n$.
\end{proof}

In a similar fashion, we can verify \cref{thm:upperboundsymmetric}.

\begin{proof}[Proof of \cref{thm:upperboundsymmetric}]
It is shown in \cite[Theorem~1.3]{GrundbacherKobos2024} that a point-symmetric convex body $K \in \CK^2$ with $d_{BM}(K,P) \geq 1 + \frac{10}{\sqrt{2}} \delta$ for a parallelogram $P \subset \R^2$ and some $\delta > 0$ satisfies $d_{BM}(K,\B_2^2) < \sqrt{2} - \delta$.
By \cref{thm:dbm_asymm}, the latter in particular means $\alpha_1(K) < \sqrt{2} - \delta$.
Therefore, $d_{BM}(K,P) \geq 1 + \varepsilon$ for some $\eps > 0$ implies with $\delta \defeq \frac{\sqrt{2}}{10} \eps$ that $\alpha_1(K) < \sqrt{2} - \delta = \sqrt{2} \left( 1 - \frac{\eps}{10} \right)$ as claimed.
\end{proof}

The final result of this section is another consequence of \cref{thm:fejestoth}.
It provides a stability improvement of the upper bound on the Minkowski chirality from \cref{thm:dbm} whenever the $j$th Minkowski chirality of simplices can be bounded away from $n$.
In this case, we also obtain an explicit improvement of the absolute upper bound $n$ on the $j$th Minkowski chirality as outlined below.

\begin{theorem}\label{thm:stability_near_simplex}
Let $K \in \CK^n$ with $\alpha_0(K) > n - \eps$ for some $\eps \in [0,\frac{1}{n})$, $j \in \{0, \ldots, n\}$, and $s(n,j)\defeq\sup \setcond{\alpha_j(T)}{T \in \CK^n \text{ simplex}}$.
Then
\begin{equation*}
    \alpha_j(K) < s(n,j) \left( 1 + \frac{(n+1)\eps}{1-n\eps} \right).
\end{equation*}
\end{theorem}

Note that by \cref{thm:triangle}, the above supremum $s(n,j)$ is not always attained as a maximum.

\begin{proof}
\cref{thm:fejestoth} yields for any simplex $T \in \CK^n$ the inequality
\begin{equation*}
    \alpha_j(K)
    \leq \alpha_j(T) \frac{R(K,T)}{r(K,T)}
    \leq s(n,j) \frac{R(K,T)}{r(K,T)}.
\end{equation*}
Minimizing the right-hand term over all choices of $T$ thus yields $\alpha_j(K) \leq s(n,j) d_{BM}(K,T)$.
Lastly, we apply \cite[Theorem~2.1]{Schneider2009}, which states that
\begin{equation*}
    d_{BM}(K,T)
    < 1 + \frac{(n+1)\eps}{1-n\eps}
\end{equation*}
under the assumptions in the theorem.
\end{proof}

Finally, let us discuss the improvement of the absolute upper bound $n$ on the $j$th Minkowski chirality.
Whenever $s(n,j) < n$, there exists a unique number $\eps(n,j) \in (0,\frac{1}{n})$ that satisfies
\begin{equation*}
    \sqrt{n (n - \eps(n,j))}
    = s(n,j) \left( 1 + \frac{(n+1) \eps(n,j)}{1-n\eps(n,j)} \right).
\end{equation*}
By \eqref{eq:improved_upperbound} and \cref{thm:stability_near_simplex}, any $K \in \CK^n$ satisfies $\alpha_j(K) \leq c(n,j) \defeq \sqrt{n (n - \eps(n,j))}$ since either $\alpha_0(K) \leq n - \eps(n,j)$ and thus $\alpha_j(K) \leq \sqrt{ \alpha_0(K)n} \leq c(n,j)$, or $\alpha_j(K) > n - \eps(n,j)$ and thus 
\begin{equation*}\alpha_j(K) < s(n,j) \left( 1 + \frac{(n+1) \eps(n,j)}{1-n\eps(n,j)} \right) = c(n,j).
\end{equation*}

For $(n,j) = (2,1)$, \cref{thm:triangle} shows $s(2,1) = \sqrt{2}$.
We may therefore compute $c(2,1)$ explicitly, which leads to 
\begin{equation}\label{eq:c21}
    \alpha_1(K) \leq \sqrt{ \frac{1}{6} \left( 13 - \frac{11}{\left( 631 + 54 \sqrt{137} \right)^{\nicefrac{1}{3}}} + \left( 631 + 54 \sqrt{137} \right)^{\nicefrac{1}{3}} \right)} < 1.95
\end{equation}
for $K\in\CK^2$.

\section{Parallelograms}\label{chap:planar}

In this section, we prove \cref{thm:parallelogram}.
Let us describe the roadmap of our proof.

First, note that \cref{thm:dbm} already verifies that $\alpha_1(K)\geq 1$, with equality precisely if $K$ is a rectangle or a rhombus.
Therefore, it remains to prove \cref{thm:parallelogram} for parallelograms that are neither rectangles nor rhombuses.
Such parallelograms shall be referred to as \emph{\para{}} in the sequel. In \cref{thm:cool-proof}, we show that a reflection axis at which $\alpha_1(K)$ is attained is necessarily parallel or perpendicular to the bisector of an angle formed by consecutive edges or by the diagonals of $K$, or to a principal axis of the John ellipse $\John(K)$.

Second, we derive the values of $R(K,\Phi_U(K))$ for the candidate subspaces $U$ in \cref{prop:angle-bisector-parallelogram}.
To do so, note that by \cref{thm:similarity-invariance}, we may parameterize parallelograms by their \emph{side length ratio} $r\geq 1$, i.e., the ratio of the length of the longer edge divided by the length of the shorter one, and the larger of the two interior angles $\theta\in [\frac{\uppi}{2},\uppi)$.
In particular, this means that $r>1$ if $K$ is not a rhombus, and $\theta>\frac{\uppi}{2}$ if $K$ is not a rectangle.

Third, we prove \eqref{eq:parallelogram-range} and the characterization of its boundary cases in \cref{thm:range-parallelogram}.
(Note that the bound $\alpha_1(K)\leq \sqrt{2}$ for parallelograms $K$ already follows from \cref{thm:dbm}.)

Last, we describe a \enquote{phase diagram} for the optimal reflection axes in \cref{thm:parallelogram-phase-diagram}, i.e., we establish explicit conditions on $r$ and $\theta$ that tell us which of the candidates for the optimal reflection axis is the one that attains $\alpha_1(K)$.

The following result is based on the fact that for any set $K\subset\R^2$ and any two subspaces $U,U'\in\gr{1}{\R^2}$, the set $\Phi_{U'}(K)$ is the image of $\Phi_U(K)$ under a rotation at the coordinate origin.
This idea is used again to determine the optimal reflection-axes for triangle in the next section. 

\begin{proposition}\label{thm:cool-proof}
Let $K \in \CK^2$ be a \para{} $0$-symmetric parallelogram and $U\in\gr{1}{\R^2}$ such that $R(K,\Phi_U(K)) = \alpha_1(K)$.
Then $U$ is parallel or perpendicular to either
\begin{enumerate}[label={(\roman*)},leftmargin=*,align=left,noitemsep]
\item{the bisector of an angle formed by consecutive edges of $K$,\label{optimal-edge-bisector}}
\item{the bisector of an angle formed by the diagonals of $K$,\label{optimal-diagonal-bisector} or}
\item{a principal axis of the John ellipsoid of $K$.\label{optimal-john}}
\end{enumerate}
\end{proposition}
\begin{proof}
Let $K=-K$ and $C=-C$ be parallelograms such that $K\subset C$ and $R(K,C)=1$.
By \cref{thm:brandenberg-koenig}, either two vertices of $K$ that are the endpoints of one of its diagonals lie on the boundary of $C$, or all four vertices of $K$ do.
Likewise, by $0$-symmetry, there are either two or four edges of $C$ that have nonempty intersection with $K$.

We do a case distinction for these two numbers.
In each case, we determine the implications of the additional assumption that $C=\lambda \Phi_U(K)$ for some $\lambda>1$ and some $U\in\gr{1}{\R^2}$, and specifically that $U$ is such that $\lambda=\alpha_1(K)$.

\emph{Case 1: There are precisely two vertices of $K$ on the boundary of $C$ and precisely two edges of $C$ have nonempty intersection with $K$.}
This means that there are two vertices $v_1$ and $-v_1$ of $K$ in the relative interior of edges $F$ and $-F$ of $C$, respectively, while the remaining two vertices $v_2$ and $-v_2$ of $K$ are elements of $\inte(C)$.
Then there is a rotation $\Psi:\R^2\to\R^2$ around the coordinate origin (with a small rotation angle in the appropriate direction) such that $\Psi(\pm v_1),\Psi(\pm v_2)\in \inte(C)$, i.e., $K\subset \inte(\Psi^{-1}(C))$ and hence $R(K,\Psi^{-1}(C))<1$.
The direction of the rotation angle corresponds to where the smaller angle formed by $\aff\setn{-v_1, v_1}$ and a perpendicular to $F$ is, and the existence of a rotation angle sufficiently close to $0$ follows by continuity of rotations.
If $C=\lambda \Phi_U(K)$ for some $\lambda$ and $U$, then there exists $U'\in\gr{1}{\R^2}$ such that $\Psi^{-1}(C)=\lambda\Phi_{U'}(K)$, which shows that our assumption $\lambda=R(K,\Phi_U(K))=\alpha_1(K)$ is violated.

\emph{Case 2: There are precisely two vertices of $K$ on the boundary of $C$ and all four edges of $C$ have nonempty intersection with $K$.}
This means that $K$ and $C$ share precisely one diagonal.
If $C=\lambda \Phi_U(K)$ for some $\lambda$ and $U$ and if the shared diagonal was the long or the short diagonal in both $K$ and $C$, then the shared diagonal would be a subset of $U$ (or of its orthogonal complement) and the interior angles at the shared vertices would be the same for $K$ and $C$.
By $K\subset C$, this would already imply that $K=C$, but since $K$ is assumed not to be a rhombus, this contradicts the assumption that $C$ is a dilated mirror image of $K$.
Thus, for $K$ and $C=\lambda \Phi_U(K)$ to be \para{} parallelograms that share precisely one diagonal, the reflection axis $U$ must be the bisector of one of the angles formed by the diagonals, see \cref{fig:P1}.

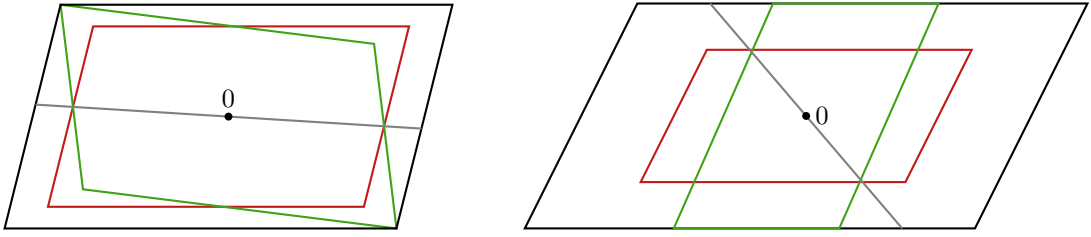
\begin{figure}[ht]
\newcommand\scale{0.6}
\newcommand\ds{1.5/\scale pt}
\centering
	\begin{tikzpicture}[scale=\scale]
    \draw [thick, dred] (4,2) -- (-3,2) -- (-4,-2) -- (3,-2) -- cycle;
    \draw [thick, dgreen] (-3.721,2.4807) -- (3.2249,1.6125) -- (3.721,-2.4807) -- (-3.2249,-1.6125) -- cycle;
    \draw [thick] (4*1.2403,2*1.2403) -- (-3*1.2403,2*1.2403) -- (-4*1.2403,-2*1.2403) -- (3*1.2403,-2*1.2403) -- cycle;
    \draw[thick, gray] (-4.274682681, 0.2661321184)-- (4.274682681,-0.2661321184);
    \fill (0,0) circle (\ds) node[above]{$0$};
\end{tikzpicture}
\qquad
\renewcommand\scale{0.44}
\begin{tikzpicture}[scale=\scale]
    \draw [thick, dred] (5,2) -- (-3,2) -- (-5,-2) -- (3,-2) -- cycle;
    \draw [thick] (5*1.7,2*1.7) -- (-3*1.7,2*1.7) -- (-5*1.7,-2*1.7) -- (3*1.7,-2*1.7) -- cycle;
    \draw [thick, dgreen] (4,2*1.7) -- (-1,2*1.7) -- (-4,-2*1.7) -- (1,-2*1.7) -- cycle;
    \draw[thick, gray] (-2.9,2*1.7) -- (2.9,-2*1.7); 
    
    \fill (0,0) circle (\ds) node[right] {$0$}; 
\end{tikzpicture}

  \caption{Case 2 (left) and Case 3 (right) in the proof of \cref{thm:cool-proof}: Parallelogram $K$ (green), reflection axis $U$ (grey), $\Phi_U(K)$ (red), and $C=R(K,\Phi_U(K))\Phi_U(K)$ (black).}
\label{fig:P1}
\end{figure}

\emph{Case 3: All four vertices of $K$ lie on the boundary of $C$ but only two edges of $C$ have nonempty intersection with $K$.}
This means that there are edges $F$ of $K$ and $F'$ of $C$ such that $F$ is a subset of the relative interior of $F'$.
If $C=\lambda \Phi_U(K)$ for some $\lambda$ and $U$ and if $F$ and $F'$ are both long or both short edges of $K$ and $C$, respectively, then $U$ is parallel or perpendicular to the common direction of $F$ and $F'$.
But then $F\subset F'$ would imply that $K$ was a rectangle.
Therefore, the reflection $\Phi_U$ necessarily maps the directions of the edges of $K$ onto each other, i.e., $U$ is parallel or perpendicular to the bisector of the angle formed by a pair of consecutive edges of $K$, see \cref{fig:P1} for an illustration.

\emph{Case 4: All four vertices of $K$ lie on the boundary of $C$ and all four edges of $C$ have nonempty intersection with $K$.}
If $K$ and $C$ share one diagonal and if $C$ is a dilated mirror image of $K$, then $U$ is parallel or perpendicular to the shared diagonal, or to the bisector of the angle formed by the diagonals of $K$.
Either way, all four vertices of $K$ lying on the boundary of $C$ then contradicts the assumption of $K$ being \para{}.
Hence $K$ and $C$ do not share any vertices, i.e., every edge of $C$ contains precisely one vertex of $K$ in its relative interior.
It remains to show that then $U$ is one of the principal axes of $\John(K)$.

Let $A\in\gl(\R^2)$ such that $A(\John(K)) = \B_2^2$.
By \cite[p.~203]{John1948}, $A(K)$ is a square with all vertices on $\bd(\sqrt{2}\B_2^2)$.
Let $u^1, u^2 \in \R^2$ be the standard Euclidean unit vectors and assume without loss of generality that $A(K)=\sqrt{2}\conv\setn{u^1,u^2,-u^1,-u^2}$.
There exist vertices $v,w$ of $\frac{1}{\alpha_1(K)} A(\Phi_U(K))$ with $v\in \sqrt{2}(u^1,u^2)$ and $w\in \sqrt{2}(-u^1,u^2)$.
Denote by $q^1$ the intersection point of $[0,\sqrt{2}u^1]$ and $[v,-w]$, and by $q^2$ the intersection point of $[0,\sqrt{2}u^2]$ and $[v,w]$, see \cref{fig:case3parallelotope}.
For having a vertex of $A(K)$ in the relative interior of each edge of $\alpha_1(K) A(\Phi_U(K))$, we need $\gauge{q^1}=\gauge{q^2}$.
Let $\lambda, \mu \in (0,1)$ with $v=\lambda \sqrt{2}u^1 +(1-\lambda)\sqrt{2}u^2 $ and $w=(1-\mu)(-\sqrt{2}u^1)+\mu \sqrt{2}u^2$.

The line through $v$ and $w$ is defined by the equation
\begin{equation*}
    x_2=\frac{1-\lambda -\mu}{1-\mu+\lambda} x_1+\sqrt{2}\left(\mu +(1-\mu)\frac{1-\lambda-\mu}{1-\mu+\lambda} \right),
\end{equation*}
which can be easily verified by checking that it is satisfied by $v$ and $w$.
This immediately implies 
\begin{equation*}
    \gauge{q^2}=\sqrt{2}\left(\mu +(1-\mu)\frac{1-\lambda-\mu}{1-\mu+\lambda} \right).
\end{equation*}
Analogously, we can compute
\begin{align*}
    \gauge{q^1}=\sqrt{2}\left(\lambda +(1-\lambda)\frac{1-\lambda-\mu}{1+\mu-\lambda} \right).
\end{align*}
The condition $\gauge{q^1} = \gauge{q^2}$ can be rearranged to
\begin{equation*}
    2 \sqrt{2} \frac{((1 - \lambda)(1 - \mu) + \lambda \mu) (\lambda - \mu)}{(1-\lambda+\mu)(1-\mu+\lambda)} = 0.
\end{equation*}
Since $\lambda, \mu \in (0,1)$, this implies $\lambda = \mu$.
Consequently $\frac{1}{\alpha_1(K)}A(\Phi_U(K))$ is a square.
Therefore, $\B_2^2$ is also the John ellipsoid of $A(\Phi_U(K))$, and the affine equivariance of John ellipsoids shows that $\John(K)=\John(\Phi_U(K))=\Phi_U(\John(K))$.
This means that $U$ is a principal axis of $\John(K)$.
\begin{figure}

    \noindent \begin{minipage}{0.45\textwidth}
        \centering
       \begin{tikzpicture}[scale=1.2]
\tkzDefPoints{0/0/O,0/1/P,1.41421/0/A,-1.41421/0/C,0/1.41421/B,0/-1.41421/D}
\tkzDefPointOnLine[pos=0.2](A,B)
\tkzGetPoint{A1}
\tkzDefPointOnLine[pos=0.4](B,C)
\tkzGetPoint{B1}
\tkzDefPointOnLine[pos=0.2](C,D)
\tkzGetPoint{C1}
\tkzDefPointOnLine[pos=0.4](D,A)
\tkzGetPoint{D1}
             \tkzDrawCircle(O,A)
             \tkzDrawCircle(O,P)
     \tkzInterLL(O,B)(B1,A1)
     \tkzGetPoint{L}
     \tkzInterLL(O,A)(A1,D1)
     \tkzGetPoint{M}
     \tkzGetVectxy(O,L){V}
      \tkzDefPointBy[homothety=center O ratio 1.41421/\Vy](A1)
      \tkzGetPoint{A2}
      \tkzDefPointBy[homothety=center O ratio 1.41421/\Vy](B1)
      \tkzGetPoint{B2}
      \tkzDefPointBy[homothety=center O ratio 1.41421/\Vy](C1)
      \tkzGetPoint{C2}
      \tkzDefPointBy[homothety=center O ratio 1.41421/\Vy](D1)
      \tkzGetPoint{D2}
      \tkzDrawPolygon[thick, dashed, dred](A2,B2,C2,D2)
              \tkzDrawPolygon[dgreen, thick](A,B,C,D)
              \tkzDrawPolygon[dred, dotted, thick](A1,B1,C1,D1)
              \tkzDrawPoints(A,B,C,D,A1,B1,C1,D1,O,L,M)
\tkzLabelPoint[below](O){$0$}
\tkzLabelPoint[above](A1){$v$}
\tkzLabelPoint[above](B1){$w$}
\tkzLabelPoint[left](M){$q^1$}
\tkzLabelPoint[below](L){$q^2$}
              
        \end{tikzpicture}
    \end{minipage}
    \begin{minipage}{0.4\textwidth}
    \centering
        \begin{tikzpicture}[scale=1.2]
\tkzDefPoints{0/0/O,0/1/P,1.41421/0/A,-1.41421/0/C,0/1.41421/B,0/-1.41421/D}
\tkzDefPointBy[rotation= center O angle 30 ](A)
\tkzGetPoint{A1}
\tkzDefPointBy[rotation= center O angle 30 ](B)
\tkzGetPoint{B1}
\tkzDefPointBy[rotation= center O angle 30 ](C)
\tkzGetPoint{C1}
\tkzDefPointBy[rotation= center O angle 30 ](D)
\tkzGetPoint{D1}
             \tkzDrawCircle(O,A)
             \tkzDrawCircle(O,P)
     \tkzInterLL(O,A)(A1,D1)
     \tkzGetPoint{L}
     \tkzGetVectxy(O,L){V}
      \tkzDefPointBy[homothety=center O ratio 1.41421/\Vx](A1)
      \tkzGetPoint{A2}
      \tkzDefPointBy[homothety=center O ratio 1.41421/\Vx](B1)
      \tkzGetPoint{B2}
      \tkzDefPointBy[homothety=center O ratio 1.41421/\Vx](C1)
      \tkzGetPoint{C2}
      \tkzDefPointBy[homothety=center O ratio 1.41421/\Vx](D1)
      \tkzGetPoint{D2}
      \tkzDrawPolygon[thick, dashed, dred](A2,B2,C2,D2)
              \tkzDrawPolygon[dgreen, thick](A,B,C,D)
              \tkzDrawPolygon[dred, thick](A1,B1,C1,D1)
              \tkzDrawPoints(A,B,C,D,A1,B1,C1,D1,O)
\tkzLabelPoint[below](O){$0$}
              
        \end{tikzpicture}
  
    \end{minipage}
    \caption{Case 4 in the proof of \cref{thm:cool-proof}: 
    Let all edges of $A(K)$  (green) contain a vertex of a parallelogram $P$ (red, dotted) that is not a square.
    Then $P$ cannot be a dilatation of $A(\Phi_U(K))$ since not all edges of $R(A(K),P)P$ (red, dashed) contain a vertex of $A(K)$ in their relative interiors (right panel). If $U$ is a principal axis of the John ellipsoid, $A(K)$ (green) and $A(\Phi_U(K))$ (red) are squares, both with John ellipsoid $\B_2^2$ and vertices on $\bd(\sqrt{2}\B_2^2)$. In the relative interior of each
    edge of $R(K,\Phi_U(K)) A(\Phi_U(K))$, there exists a vertex of  $A(K)$ (left panel). \label{fig:case3parallelotope}}
\end{figure}
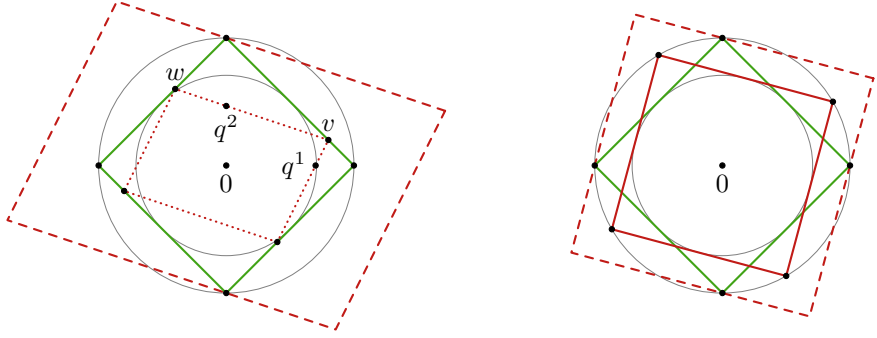
\end{proof}

When $U$ is one of the  candidate subspaces from \cref{thm:cool-proof}, the optimal containment situation $K\subset R(K,\Phi_U(K))\Phi_U(K)$ can be characterized in terms of the location of the vertices of $K$ relative to the boundary of $R(K,\Phi_U(K))\Phi_U(K)$.
In case of $U$ being a principal axis of the John ellipse, this characterization crucially simplifies the proof of \cref{prop:angle-bisector-parallelogram}\ref{dilation-john} below.

\begin{proposition}\label{thm:cool-proof2}
Let $K \in \CK^2$ be a \para{} 
$0$-symmetric parallelogram and $U\in\gr{1}{\R^2}$.
Then the following statements are true:
\begin{enumerate}[label={(\roman*)},leftmargin=*,align=left,noitemsep]
\item{$U$ is parallel to the bisector of an angle formed by consecutive edges of $K$ if and only if an edge of $K$ is a subset of an edge of $R(K,\Phi_U(K))\Phi_U(K)$,\label{touch-edge-bisector}}
\item{$U$ is parallel to the bisector of an angle formed by 
the diagonals of $K$ if and only if $K$ and $R(K,\Phi_U(K))\Phi_U(K)$ share a vertex, and\label{touch-diagonal-bisector}}
\item{$U$ is parallel to a principal axis of $\John(K)$ if and only if every edge of $R(K,\Phi_U(K))\Phi_U(K)$ contains a vertex of $K$ in its relative interior.\label{touch-john}}
\end{enumerate}
\end{proposition}

\begin{proof}
The sufficiency of the conditions on the location of the vertices of $K$ relative to the boundary of $R(K,\Phi_U(K))\Phi_U(K)$ has been addressed in the proof of \cref{thm:cool-proof}.
Their necessity is clear for \labelcref{touch-edge-bisector,touch-diagonal-bisector}.
It remains to show the necessity for \labelcref{touch-john}.

Let $U$ be a principal axis of $\John(K)$.
\cref{prop:ellipsoid_invar_reflection} and the equivariance of the John ellipse show $\John(K)=\Phi_U(\John(K))=\John(\Phi_U(K))$.
By \cite[p.~203]{John1948}, all vertices of $K$ and $\Phi_U(K)$ lie on $\bd(\sqrt{2}\John(K))$.
Let $A\in\gl(\R^n)$ be such that $ A(\John(K)) = \B_2^2$.
Then $\B_2^2$ is the John ellipse of the parallelograms $A(K)$ and $A(\Phi_U(K))$, and $A(K)$ and $A(\Phi_U(K))$ are $0$-symmetric squares with all vertices lying on $\bd(\sqrt{2}\B_2^2)$.
Since $\alpha_1(K) > 1$, the squares $A(K)$ and $A(\Phi_U(K))$ are different from each other.

We note for a vertex $w$ of $A(K)$ and $p_w \in [0, w] \cap \bd(A(\Phi_U(K)))$ that the ratio $\frac{\gauge{w}}{\gauge{p_w}} > 1$ does not depend on the choice of $w$.
See the right panel of \cref{fig:case3parallelotope} again for an illustration.
Thus, we have $A(K) \subset \frac{\gauge{w}}{\gauge{p_w}} A(\Phi_U(K))$ and there is a vertex of $A(K)$ in the relative interior of each edge of $\frac{\gauge{w}}{\gauge{p_w}}A(\Phi_U(K))$.
We have $\frac{\gauge{w}}{\gauge{p_w}} = R(A(K),A(\Phi_U(K))) = R(K,\Phi_U(K))$, so applying $A^{-1}$ to return to $K$ and $R(K,\Phi_U(K))\Phi_U(K)$ completes the proof.
\end{proof}

Next, we determine for a parallelogram $K \in \CK^2$ the minimal dilation factors in an inclusion $K \subset x + \lambda \Phi_U(K)$, where $U$ is the bisector of an angle formed by consecutive edges of $K$ or the diagonals of $K$, or a principal axis of the John ellipse of $K$.

\begin{proposition}\label{prop:angle-bisector-parallelogram}
Let $K\in\CK^2$ be a parallelogram with side length ratio $r> 1$ and larger interior angle $\theta \in (\frac{\uppi}{2},\uppi)$.
Then the following statements are true: \begin{enumerate}[label={(\roman*)},leftmargin=*,align=left,noitemsep]
    \item{If $U$ is the bisector of an interior angle of $K$, then
    \begin{equation*}
    R(K,\Phi_U(K))=r.
    \end{equation*}\label{dilation-bisector}}
    \item{If $U$ is an angle bisector of the diagonals of $K$, then
    \begin{equation*}
    R(K,\Phi_U(K))=\frac{r^2-2r\cos(\theta)+1}{\sqrt{(r^2+1)^2-4r^2\cos(\theta)^2}}.
    \end{equation*}\label{dilation-diagonal}}
    \item{If $U$ is a principal axis of the John ellipse of $K$, then
    \begin{equation*}
    R(K,\Phi_U(K))=\frac{r^2-2r\cos(\theta)-1}{\sqrt{(r^2-1)^2+4r^2\cos(\theta)^2}}.
    \end{equation*}\label{dilation-john}}
\end{enumerate}
\end{proposition}

\begin{proof}
By \cref{thm:similarity-invariance}, we may assume $K=\conv\setn{a,b,-a,-b}$, where $a=(1,0)^\top$ and $b=(z_1,z_2)^\top\in\R^2$ with $z_1,z_2>0$ and $z_1^2+z_2^2<1$.

For \ref{dilation-bisector}, if $U$ is the angle bisector of $K$ at $a$, then $\Phi_U$ maps $\aff\setn{a,b}$ onto $\aff\setn{a,-b}$ and vice versa.
In this case,
$K\subset a+\max\setn{\frac{\gauge{a-b}}{\gauge{a+b}},\frac{\gauge{a+b}}{\gauge{a-b}}} (\Phi_U(K)-a)$ is optimal by \cref{thm:brandenberg-koenig}, see left panel of \cref{fig:parallelogram-bisector-reflection}, i.e., $R(K,\Phi_U(K))=\max\setn{\frac{\gauge{a-b}}{\gauge{a+b}},\frac{\gauge{a+b}}{\gauge{a-b}}}=r$.

For \ref{dilation-diagonal}, note that $\Phi_U$ maps $\aff\setn{a,-a}$ onto $\aff\setn{b,-b}$ and vice versa.
For the ratio of the diagonal lengths $\lambda=\max\setn{\frac{2\gauge{a}}{2\gauge{b}},\frac{2\gauge{b}}{2\gauge{a}}}$, the containment $K\subset \lambda \Phi_U(K)$ is optimal by \cref{thm:brandenberg-koenig}, see right panel of \cref{fig:parallelogram-bisector-reflection}.
Thus, $R(K,\Phi_U(K))$ equals the ratio of the diagonal lengths.
In a parallelogram with side lengths $1$ and $r$ and larger interior angle $\theta$, we may use the law of cosines to determine the lengths of the diagonals.
These turn out to be $\sqrt{r^2+1-2r\cos(\theta)}$ and $\sqrt{r^2+1+2r\cos(\theta)}$, where the former is the larger one since $\cos(\theta) < 0$.

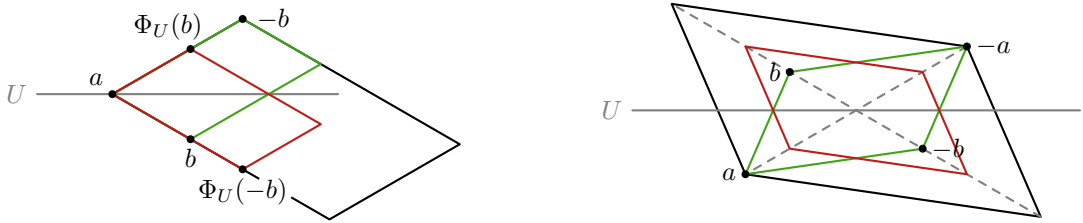
\begin{figure}[h]
 \begin{minipage}{0.47\textwidth}
    \centering
\begin{tikzpicture}[line cap=round,line join=round,>=stealth,x=1.0cm,y=1.0cm]
\pgfmathsetmacro{\x}{1.5}
\coordinate (c) at (30:2);
\coordinate (b) at (-30:1.2);
\coordinate (a) at (0,0);
\coordinate (cprime) at (-30:2);
\coordinate (bprime) at (30:1.2);
\coordinate (d) at (-30:3.3333);
\draw[thick] (a)--(c)--($(c)+(d)$)--(d)--cycle;
\draw[gray,thick] (3,0)--(-1,0) node[left]{$U$};
\draw[dgreen,thick] (c)--($(c)+(b)$)--(b)--(a)--cycle;
\draw[dred,thick] (cprime)--($(cprime)+(bprime)$)--(bprime)--(a)--cycle;
\fill[white] (-30:2.5) circle(7pt);
\fill (c) circle (1.5pt) node[right=2pt]{$-b$};
\fill (b) circle (1.5pt) node[below]{$b$};
\fill (a) circle (1.5pt) node[above left]{$a$};
\fill (cprime) circle (1.5pt) node[below]{$\Phi_U(-b)$};
\fill (bprime) circle (1.5pt);
\draw ($(bprime)+(-0.3,0.3)$) node{$\Phi_U(b)$};
\end{tikzpicture}
  \end{minipage} \hfill
  \begin{minipage}{0.47\textwidth}
    \centering
\newcommand\scale{0.85}
\newcommand\ds{1.5/\scale pt}
\begin{tikzpicture}[line cap=round,line join=round,>=stealth,x=1.0cm,y=1.0cm, scale=\scale]
\pgfmathsetmacro{\x}{1.5}
\coordinate (c) at (30:2);
\coordinate (c2) at (30:-2);
\coordinate (b) at (-30:1.2);
\coordinate (b2) at (-30:-1.2);
\coordinate (a) at (0,0);
\coordinate (cprime) at (-30:2);
\coordinate (bprime) at (30:1.2);
\coordinate (cprime2) at (-30:-2);
\coordinate (bprime2) at (30:-1.2);

\coordinate (d) at (-30:3.3333);
\coordinate (d2) at (-30:-3.3333);
\draw[thick] (c)--(d)--(c2)--(d2)--cycle;
\draw[gray,thick,dashed] (c)--(c2);
\draw[gray,thick,dashed] (d)--(d2);
\draw[gray,thick] (3.5,0)--(-3.5,0) node[left]{$U$};
\draw[dgreen,thick] (c)--(b2)--(c2)--(b)--cycle;
\draw[dred,thick] (cprime)--(bprime2)--(cprime2)--(bprime)--cycle;
\fill (c) circle (\ds) node[right]{$-a$};
\fill (b) circle (\ds) node[right]{$-b$};
\fill (b2) circle (\ds) node[left]{$b$};
\fill (c2) circle (\ds) node[left]{$a$};
\end{tikzpicture}
  \end{minipage}
  \caption{Optimal containment of a parallelogram $K$ (green) in the dilated mirror image after reflection across the reflection axis $U$ from the proof of \cref{prop:angle-bisector-parallelogram} for the case of a bisector of an interior angle (left panel) and a bisector of the diagonals (right panel): reflection axis $U$ (grey), $\Phi_U(K)$ (red), and an appropriate translate of $R(K,\Phi_U(K))\Phi_U(K)$ (black).}
  \label{fig:parallelogram-bisector-reflection}
\end{figure}

For \ref{dilation-john}, we invoke a generalization of Marden's theorem from \cite[Proposition~5]{Parish2006}, which implies that upon identification of $\R^2$ with $\C$, the focal points of $\John(K)$ are solutions $w\in\C$ of the equation $w^2=\frac{1+z_1^2-z_2^2}{2}+\ii z_1z_2$.
By assumption, the real and the imaginary parts of $\frac{1+z_1^2-z_2^2}{2}+\ii z_1z_2$ are both positive, so there is one solution $w$ with argument in $(0,\frac{\uppi}{4})$.
Therefore, the major axis of $\John(K)$, i.e., the straight line containing both focal points, is given by $U=\setcond{\mu (\cos(\varphi),\sin(\varphi))^\top
}{\mu\in\R}$, where $\varphi \in (0,\frac{\uppi}{4})$ is determined by $\cot(2\varphi)=\frac{1+z_1^2-z_2^2}{2z_1z_2}$.
Setting $t\defeq \sin(2\varphi)$, we obtain 
\begin{equation*}
    \frac{\sqrt{1-t^2}}{t}=\frac{1+z_1^2-z_2^2}{2z_1z_2}.
\end{equation*}
Squaring and solving for $t$ leads to
\begin{equation*}
    \sin(2\varphi)=\frac{2z_1z_2}{\sqrt{(1+z_1^2-z_2^2)^2+(2z_1z_2)^2}} \qquad\text{and}\qquad     \cos(2\varphi)=\frac{1+z_1^2-z_2^2}{\sqrt{(1+z_1^2-z_2^2)^2+(2z_1z_2)^2}}.
\end{equation*}
Note that $\Phi_U(a) = (\sin(2\varphi),\cos(2\varphi))^\top$.
By \cref{thm:cool-proof2}, all vertices of $\frac{1}{R(K,\Phi_U(K)}\Phi_U(K)$ are boundary points of $K$.
Let us thus compute the unique number $\lambda>0$ for which $\frac{1}{\lambda}\Phi_U(a)\in\bd(K)$.

With
\begin{equation*}
    \mu
    \defeq \frac{2 z_1}{1 - z_1^2 - z_2^2 + 2 z_1}
    = 1 - \frac{1 - z_1^2 - z_2^2}{1 - z_1^2 - z_2^2 + 2 z_1} \in (0,1),
\end{equation*}
we compute
\begin{align*}
    [a,b]
    & \ni (1-\mu) a + \mu b
    = \left( (1-\mu) + \mu z_1, \mu z_2 \right)^\top
    = \left( \frac{1 + z_1^2 - z_2^2}{1 - z_1^2 - z_2^2 + 2 z_1}, \frac{2 z_1 z_2}{1 - z_1^2 - z_2^2 + 2 z_1} \right)^\top
    \\
    & = \frac{\sqrt{(1 + z_1^2 - z_2^2)^2 + (2 z_1 z_2)^2}}{1 - z_1^2 - z_2^2 + 2 z_1} \left( \sin(2\varphi), \cos(2 \varphi) \right)^\top.
\end{align*}
Thus, we have for
\begin{equation} \label{eq:john-lambda}
    \lambda
    \defeq \frac{1 - z_1^2 - z_2^2 + 2 z_1}{\sqrt{(1 + z_1^2 - z_2^2)^2 + (2 z_1 z_2)^2}}
    > 0
\end{equation}
that $\frac{1}{\lambda} \Phi_U(a) \in [a,b] \subset \bd(K)$.
Since there exists only one such $\lambda > 0$, this yields $R(K,\Phi_U(K)) = \lambda$.

It remains to express $\lambda=R(K,\Phi_U(K))$ in terms of the side length ratio $r>1$ and the interior angle $\theta\in (\frac{\uppi}{2},\uppi)$.
Since $z_1,z_2>0$ and $z_1^2+z_2^2<1$, the line segment $[a,b]$ is one of the short edges of $K$ and the interior angle of $K$ at $b$ is obtuse.
If we abbreviate the length of this segment by $s\defeq \gauge{b-a}$, then $\gauge{b-(-a)}=rs$ is the other side length of $K$.
The law of cosines for the triangle $\conv\setn{a,b,-a}$ implies $r^2s^2+s^2-2rs^2\cos(\theta)=4$ or, equivalently,
\begin{equation*}
s^2=\frac{4}{1+r^2-2r\cos(\theta)}.
\end{equation*}
Note that $1+r^2-2r\cos(\theta)>0$ since $\cos(\theta)<0$.
Now, let $c=(z_1,0)^\top \in [-a,a]$ be the foot of the perpendicular from $b$ to $[-a,a]$.
Applying Pythagoras's theorem in the triangles $\conv\setn{a,b,c}$ and $\conv\setn{-a,b,c}$ gives
$(1-z_1)^2+z_2^2=s^2$ and $(1+z_1)^2+z_2^2=r^2s^2$.
It follows that
\begin{equation}\label{eq:first-entry}
 z_1=\frac{1}{4}(r^2-1)s^2=\frac{r^2-1}{1+r^2-2r\cos(\theta)}.
\end{equation}
This implies $1-z_1=\frac{2(1-r\cos(\theta))}{1+r^2-2r\cos(\theta)}$.
We obtain \begin{equation*}
    z_2^2
    =s^2-(1-z_1)^2
    =\frac{4}{1+r^2-2r\cos(\theta)}-\frac{4(1-r\cos(\theta))^2}{(1+r^2-2r\cos(\theta))^2}
    =\frac{4r^2\sin(\theta)^2}{(1+r^2-2r\cos(\theta))^2}.
\end{equation*}
Since $z_2>0$ and $\sin(\theta)>0$, we further have
\begin{equation}\label{eq:second-entry}
    z_2=\frac{2r\sin(\theta)}{1+r^2-2r\cos(\theta)}.
\end{equation}
Plugging \labelcref{eq:first-entry,eq:second-entry} into the numerator of \eqref{eq:john-lambda} and using the \texttt{subs} and \texttt{simplify} methods in SymPy \cite{sympy}, we obtain 
\begin{equation*}
    1-z_1^2-z_2^2+2z_1 = \frac{2(r^2-2r\cos(\theta)-1)}{1+r^2-2r\cos(\theta)}.
\end{equation*}

For the expression under the square root in the denominator of \eqref{eq:john-lambda}, a similar approach gives
\begin{equation*}
(1+z_1^2-z_2^2)^2+(2z_1z_2)^2=\frac{4(r^4+2r^2\cos(2\theta)+1)}{(r^2-2r\cos(\theta)+1)^2}.
\end{equation*}
Hence
\begin{equation*}
R(K,\Phi_U(K))=\lambda=\frac{r^2-2r \cos(\theta)-1}{\sqrt{r^4+2r^2\cos(2\theta)+1}}=\frac{r^2-2r\cos(\theta)-1}{\sqrt{(r^2-1)^2+4r^2\cos(\theta)^2}}. \qedhere
\end{equation*}
\end{proof}

\begin{corollary}
For a \para{} parallelogram with side length ratio $r>1$ and larger interior angle $\theta\in (\frac{\uppi}{2},\uppi)$, we have 
\begin{equation}\label{eq:alpha1-parallelogram}
\alpha_1(K)=\min\setn{r,\frac{r^2-2r\cos(\theta)+1}{\sqrt{(r^2+1)^2-4r^2\cos(\theta)^2}},\frac{r^2-2r\cos(\theta)-1}{\sqrt{(r^2-1)^2+4r^2\cos(\theta)^2}}}.
\end{equation}
\end{corollary}

\begin{proof}
This follows from \cref{thm:cool-proof} together with \cref{prop:angle-bisector-parallelogram}.
\end{proof}

\newpage
\begin{remark}\label{rem:polarity}\hfill
\begin{enumerate}[label={(\roman*)},leftmargin=*,align=left,noitemsep]
\item{The affine equivariance of John ellipses offers a method of computing their principal axes for parallelograms $K\defeq \conv\setn{a,b,-a,-b}$, where $a=(a_1,a_2)^\top$ and $b=(b_1,b_2)^\top\in\R^2$, other than the one used in the proof of \cref{prop:angle-bisector-parallelogram}\ref{dilation-john}.
Let $f:\R^2\to\R^2$, $f(x)=Mx$ be the linear map given by left multiplication with the matrix
    \begin{equation*}
      M\defeq\begin{pmatrix}
            a_1&b_1\\a_2&b_2
        \end{pmatrix}\begin{pmatrix}
            -1&1\\\phantom{-}1&1
        \end{pmatrix}^{-1}.
    \end{equation*}
Then we have $f([-1,1]^2)=K$, so $\John([-1,1]^2)=\B_2^2$ yields $\John(K)=f(\B_2^2)$.
In other words, we have $\John(K)=\setcond{x\in \R^2}{x^\top M^{-\top}M^{-1}x\leq 1}$, which means that the principal axes of $\John(K)$ are given by the eigenvectors of the symmetric matrix
\begin{equation*}
    M^{-\top}M^{-1}=2\begin{pmatrix}a_1&a_2\\b_1&b_2\end{pmatrix}^{-1}\begin{pmatrix}a_1&b_1\\a_2&b_2\end{pmatrix}^{-1}.
\end{equation*}}
\item{\cref{thm:orthogonal-axis-central-symmetry} sheds some additional light on the optimal axes in \cref{thm:parallelogram}.
If $K \in \CK^2$ is a $0$-symmetric parallelogram, then so is $K^\circ$,
and we can canonically associate all candidate reflection axes for $\alpha_1(K)$ with those for $\alpha_1(K^\circ)$.
First, the directions of the principal axes of $\John(K^\circ)$ coincide with those of $\John(K)$.
Indeed, if $A \in \gl(\R^n)$ such that $K = A([-1,1]^2)$, then $K^\circ = A^{-\top}(C)$, where $C \defeq ([-1,1]^2)^\circ$ is a rotation of $[-1,1]^2$ by $45^\circ$ dilated by the factor $\frac{1}{\sqrt{2}}$.
Therefore, the equivariance of the John ellipsoid shows
\begin{align*}
    \John(K^\circ) &= \John(A^{-\top}(C)) = A^{-\top}(\John(C)) = \frac{1}{\sqrt{2}} A^{-\top}(\B_2^2)\\
    & = \frac{1}{\sqrt{2}} (A(\John([-1,1]^2))^\circ= \frac{1}{\sqrt{2}} (\John(K))^\circ.
\end{align*}
Since the principal axes of a $0$-symmetric ellipsoid and its polar coincide (see, e.g., \cite[Lemma~2.1]{GrundbacherKobos2024}), this verifies that the principal axes of $\John(K^\circ)$ and $\John(K)$ really are the same.
Second, note that the edges of $K^\circ$ are perpendicular to the diagonals of $K$ and vice versa.
With this, it is easy to see that the interior angles of $K^\circ$ can be paired up with the angles formed by the diagonals of $K$ such that the angle pairs each add up to $180^\circ$.
In particular, the interior angles of $K^\circ$ coincide with the angles formed by the diagonals of $K$, and so do the bisectors of these respective angles (up to translation).
Last, we also see that the condition that the interior angles of $K$ coincide with the angles formed by the diagonals of $K$ can be equivalently stated as $K$ and $K^\circ$ being similar to each other.
(Note that the interior angles and the angles formed by the diagonals determine a parallelogram up to similarity.)
In this case, the bisectors $U^1$ and $U^2$ of an interior angle of $K$ and an angle formed by the diagonals of $K$, respectively, must satisfy $R(K,\Phi_{U^1}(K)) = R(K,\Phi_{U^2}(K))$.
}
\end{enumerate}
\end{remark}

We are ready to prove \eqref{eq:parallelogram-range} and determine the parallelograms with the maximum Minkowski chirality.
\begin{proposition}\label{thm:range-parallelogram}
We have $\setcond{\alpha_1(K)}{K\in\CK^2\text{ is a \para{} parallelogram}}=(1,\sqrt{2}]$.
Moreover, for a \para{} parallelogram $K\in\CK^2$, the condition $\alpha_1(K)=\sqrt{2}$ is equivalent to the condition that the angles formed by consecutive edges of $K$ coincide with the angles formed by its diagonals and the ratio between the length of a long edge and that of a short edge being at least $\sqrt{2}$.
\end{proposition}
\begin{proof}
Denote, as before, $r$ as the ratio of a longer side by a shorter side and $\theta$ as the size of the obtuse angle between consecutive edges. In view of \eqref{eq:alpha1-parallelogram}, $\alpha_1(K)$ is the pointwise infimum of three continuous functions of $(r,\theta)\in (1,\infty)\times(\frac{\uppi}{2},\uppi)$, so it is continuous itself. For the first part of the claim, it is thus sufficient to show that the infimum of the set $\setcond{\alpha_1(K)}{K\in\CK^2\text{ is a \para{} parallelogram}}$ is $1$ and that the supremum of the set is $\sqrt{2}$.

For the infimum, we see that for every $\epsilon>0$, there is a \para{} parallelogram $K$ with side length ratio $r=1+\eps$ and thus, by \eqref{eq:alpha1-parallelogram}, also $\alpha_1(K)\leq 1+\eps$. Note that this is indeed an infimum and not a minimum since we are discussing \para{} parallelograms. 

Next, observe that 
\begin{equation}\label{eq:john-axis-ineq}
    \frac{r^2-2r\cos(\theta)-1}{\sqrt{(r^2-1)^2+4r^2\cos(\theta)^2}}\leq \sqrt{2}.
\end{equation}
Indeed, by squaring and substituting $y\defeq -2\cos(\theta)$, this is equivalent to $(r^2+yr-1)^2 \leq 2(r^2-1)^2 + 2y^2r^2$, which can be rearranged to $0 \leq r^4-2yr^3+(y^2-2)r^2+2yr+1 = \big(r(y-r)+1\big)^2$. This, together with \eqref{eq:alpha1-parallelogram}, shows $\alpha_1(K)\leq \sqrt{2}$. A particular example for a parallelogram $K$ with $\alpha_1(K)=\sqrt{2}$ is $K=\conv \setn{\pm (1,0), \pm (2,1)}$. This completes the proof of the first part of the claim. 

Now, observe that equality holds in \eqref{eq:john-axis-ineq} if and only if $r(r-y)=1$, which is in turn equivalent to $\cos(\theta)=\frac{1}{2}\big(\frac{1}{r}-r\big)$.
We claim that the latter condition is also equivalent to
\begin{equation}\label{eq:angle-bisectors-equality}
\frac{r^2-2r\cos(\theta)+1}{\sqrt{(r^2+1)^2-4r^2\cos(\theta)^2}} = r.
\end{equation}
Indeed, a straightforward computation shows for $\cos(\theta)=\frac{1}{2}\big(\frac{1}{r}-r\big)$ that
\begin{equation*}
\frac{r^2-2r\cos(\theta)+1}{\sqrt{(r^2+1)^2-4r^2\cos(\theta)^2}} = \frac{r^2-(1-r^2)+1}{\sqrt{(r^2+1)^2-(1-r^2)^2}}=\frac{2r^2}{\sqrt{4r^2}}=r.
\end{equation*}
Conversely, the numerator of the left-hand term in \eqref{eq:angle-bisectors-equality} is strictly decreasing in $\cos(\theta)$, whereas the denominator in strictly increasing in $\cos(\theta) \in (-1,0)$.
Thus, the entire left-hand term is strictly decreasing in $\cos(\theta) \in (-1,0)$, which means that $\cos(\theta)=\frac{1}{2}\big(\frac{1}{r}-r\big)$ is the only solution of \eqref{eq:angle-bisectors-equality}.
Altogether, in view of \eqref{eq:alpha1-parallelogram}, we have that $\alpha_1(K) = \sqrt{2}$ if and only if
\begin{equation*}
    \frac{r^2-2r\cos(\theta)+1}{\sqrt{(r^2+1)^2-4r^2\cos(\theta)^2}}
    = r
    \geq \sqrt{2}.
\end{equation*}

To complete the proof, it suffices to show that \eqref{eq:angle-bisectors-equality} is equivalent to the condition that $\theta$ equals the larger angle $\delta$ formed by the diagonals of $K$, see \cref{fig:diagonalbisector}.
To achieve this, let the vertices of $K$ be labeled as $a,b,c,d$, and denote the center of $K$ by $m=\frac{a+b+c+d}{4}$.
Assume that $[a,b]$, $[b,c]$, $[c,d]$, and $[d,a]$ are the edges of $K$ with $\gauge{a-b}=\gauge{c-d}=1$ and $\gauge{a-d}=\gauge{b-c}=r$. 

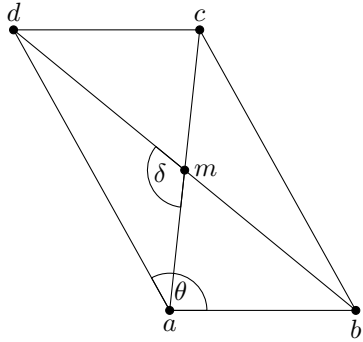
\begin{figure}[h!]
    \centering
\newcommand\scale{0.75}
\newcommand\ds{1.5/\scale pt}
\begin{tikzpicture}[line cap=round,line join=round,>=stealth,x=3.3cm,y=3.3cm, scale=\scale]

\coordinate (A) at (0,0);
\coordinate (B) at (1,0);
\coordinate (C) at (-0.8385864958115141,1.508151174209125);
\coordinate (D) at (0.16141350418848588,1.508151174209125);
\coordinate (M) at ($(B)!0.5!(C)$);
\coordinate (E) at ($(B)!0.5!(D)$);

\draw (0,0) -- (0.:0.2) arc (0.:119.07563926565886:0.2) -- cycle;
\draw (60:0.12) node{$\theta$};

\draw [shift={(M)}] (0,0) -- (140.6387408810172:0.2) arc (140.6387408810172:263.8910364471245:0.2) -- cycle;
\draw (-0.05,0.74) node{$\delta$};

\draw (A)-- (B)--(D)-- (C)--cycle;
\draw (A)-- (D);
\draw (B)-- (C);

\draw [fill=black] (M) circle (\ds) node[right]{$m$};
\draw [fill=black] (A) circle (\ds) node[below]{$a$};
\draw [fill=black] (B) circle (\ds) node[below]{$b$};
\draw [fill=black] (C) circle (\ds) node[above]{$d$};
\draw [fill=black] (D) circle (\ds) node[above]{$c$};
\end{tikzpicture}
    \caption{The obtuse angle formed by the diagonals of the parallelogram is denoted by $\delta$.}
    \label{fig:diagonalbisector}
\end{figure}

Applying the law of cosines to the triangles $\conv\setn{a,b,d}$ and $\conv\setn{a,b,c}$ gives
\begin{equation*}
    \gauge{b-d}^2=r^2+1-2r\cdot 1\cdot \cos(\theta) \qquad\text{and}\qquad \gauge{a-c}^2=r^2+1+2r\cdot 1\cdot \cos(\theta).
\end{equation*}
If we assume equality in \eqref{eq:angle-bisectors-equality}, or equivalently that, $\cos(\theta)=\frac{1}{2}\left(\frac{1}{r}-r\right)$, we get $\gauge{b-d}^2=2r^2$ and $\gauge{a-c}^2=2$.
Another application of the law of cosines in the triangle $\conv\setn{a,m,d}$ yields
\begin{equation*}
    r^2=\frac{\gauge{b-d}^2}{4}+\frac{\gauge{a-c}^2}{4}-2\frac{\gauge{b-d}}{2}\cdot\frac{\gauge{a-c}}{2}\cdot\cos(\delta),
\end{equation*}
which simplifies to $-2\cos(\delta)=r-\frac{1}{r}$.
The bijectivity of the cosine function on $(\frac{\uppi}{2},\uppi)$ shows that $\theta=\delta$.
The converse implication follows from \cref{rem:polarity}(ii).
\end{proof}

In the following, we describe a \enquote{phase diagram} for the reflection axis at which the Minkowski chirality $\alpha_1(K)$ of a given \para{} parallelogram is attained in terms of the side length ratio $r>1$ and the larger interior angle $\theta\in (\frac{\uppi}{2},\uppi)$.
This phase diagram is depicted in \cref{fig:phase-diagram-parallelogram} (left panel).

\begin{corollary}\label{thm:parallelogram-phase-diagram}
Let $K\in\CK^2$ be a \para{} parallelogram with side length ratio $r>1$ and larger interior angle $\theta\in (\frac{\uppi}{2},\uppi)$.
Define $\Psi_1,\Psi_2:(1,\infty)\to (\frac{\uppi}{2},\uppi)$ by
\begin{align*}
\Psi_1(r)&= 
\begin{cases}
\arccos(-\frac{1}{2}(r-\frac{1}{r})) & \text{if } r\in\big(1,\sqrt{2}\big),\\
\arccos(-\frac{1}{2}(\frac{-r^2+\sqrt{r^4+6r^2-7}+1}{2r})) & \text{if } r\geq\sqrt{2}
\end{cases}\\\intertext{and}
\Psi_2(r)&=
\begin{cases}
\arccos(-\frac{1}{2}(\frac{1}{r} + \sqrt{2-r^2})) & \text{if } r\in\big(1,\sqrt{2}\big),\\
\arccos(-\frac{1}{2}(\frac{-r^2+\sqrt{r^4+6r^2-7}+1}{2r})) & \text{if } r\geq\sqrt{2}.
\end{cases}
\end{align*}
Let further
\begin{align*}
\Dcal&\defeq\setcond{(r,\theta)^\top\in\R^2}{1<r,\;\frac{\uppi}{2}< \theta\leq \Psi_1(x)},\\
\Jcal&\defeq\setcond{(r,\theta)^\top\in\R^2}{1<r,\;\uppi>\theta\geq \Psi_2(x)}, \quad \text{and}\\
\Bcal&\defeq \left((1,\infty)\times \left(\frac{\uppi}{2},\uppi\right)\right)\setminus\big(\Jcal\cup\Dcal\big).
\end{align*}

Then $\alpha_1(K)=R(K,\Phi_U(K))$, where
\begin{enumerate}[label={(\roman*)},leftmargin=*,align=left,noitemsep]
\item{$U$ is the angle bisector of two consecutive edges of $K$ whenever $(r,\theta)\in \Bcal$,}
\item{$U$ is an angle bisector of the diagonals of $K$ whenever $(r,\theta)\in \Dcal$, and}
\item{$U$ is a principal axis of the John ellipse of $K$ whenever $(r,\theta)\in \Jcal$.}
\end{enumerate}
\end{corollary}

\begin{proof}
Set $y\defeq -2\cos(\theta)$ and $\Omega\defeq (1,\infty)\times (0,2)$.
Given our restrictions on $r$ and $\theta$, it is easy to see that the coordinate transform $(r,\theta) \mapsto (r,y)$ is one-to-one.
Further, define $B,D,J:\Omega\to\R$ by
\begin{equation*}
    B(r,y) = r^2,
        \quad
    D(r,y) = \frac{(r^2+ry+1)^2}{(r^2+1)^2-r^2y^2},
        \quad \text{and} \quad
    J(r,y) = \frac{(r^2+ry-1)^2}{(r^2-1)^2+r^2y^2}.
\end{equation*}
These expressions are equal to $R(K,\Phi_U(K))^2$, where $U$ is the angle bisector of two consecutive edges of $K$, an angle bisector of the diagonals of $K$, or the major axis of $\John(K)$,  respectively.
By \cref{thm:cool-proof,prop:angle-bisector-parallelogram}, proving the claim can be recast as determining the domains in $\Omega$ where each of $B$, $D$, and $J$ is minimal.
To do so, we rewrite the sets of pairs $(r,y)\in\Omega$ for which either $J(r,y)=D(r,y)$, $D(r,y)=B(r,y)$, or $J(r,y)=B(r,y)$ as graphs of functions of $r$.
(Parts of these graphs appear as solid lines in the left panel of \cref{fig:phase-diagram-parallelogram}.)
Depended on the order of these functions, we then distinguish cases to find the desired regions.

We have
\begin{align*}
&J(r,y)=D(r,y)\\
\Leftrightarrow \quad& r^6y^2+2r^5y^3+r^4y^4+r^2y^2-2r^5y-2r^4y^2+2ry = 0 \\
\Leftrightarrow \quad& r^3y^3 + 2r^4y^2+(r^5-2r^3+r)y+2-2r^4 = 0 \\
\Leftrightarrow \quad& \Big(y-\frac{-r^2-1}{r}\Big)\Big(y-\frac{-r^2-\sqrt{r^4+6r^2-7}+1}{2r}\Big)\Big(y-\frac{-r^2+\sqrt{r^4+6r^2-7}+1}{2r}\Big) = 0.
\end{align*}
From $r>1$, we know that both $\frac{-r^2-1}{r}$ and $\frac{-r^2-\sqrt{r^4+6r^2-7}+1}{2r}$ are negative.
Therefore, since we are looking for $y\in(0,2)$, the only potential solution is 
\begin{equation}\label{eq:jd}
y=\frac{-r^2+\sqrt{r^4+6r^2-7}+1}{2r}.
\end{equation}
Let us also remark that $r>1$ implies $r^4+6r^2-7=(r^2+7)(r^2-1)>(r^2-1)^2>0$.
Thus, $y$ from \eqref{eq:jd} is positive.
A direct computations shows that $\frac{-r^2+\sqrt{r^4+6r^2-7}+1}{2r}<2$ is equivalent to $(r^2-1)(r+1)>0$, which is again satisfied for $r>1$.

Thus, we have $J(r,y)=D(r,y)$ for $(r,y)^\top\in\Omega$ if and only if $(r,y)^\top$ is an element of the graph of the function $f_{JD}: (1,\infty)\to\R$ given by $f_{JD}(r)=\frac{-r^2+\sqrt{r^4+6r^2-7}+1}{2r}$.

Similarly, we get $J(r,y) < D(r,y)$ if $y>f_{JD}(r)$, and $J(r,y) > D(r,y)$ if $y<f_{JD}(r)$.

Next, consider 
\begin{align*}
&D(r,y)=B(r,y)\\
\Leftrightarrow\quad & r^4 + r^2y^2+1+2r^3y+2r^2+2ry=r^6+2r^4+r^2-r^4y^2 \\ 
\Leftrightarrow\quad & (r^2+1)(r^2y^2+2ry+1-r^4)=0 \\
\Leftrightarrow\quad & r^2y^2+2ry +1-r^4= 0 \\
\Leftrightarrow\quad & \Big(y-\Big(-r-\frac{1}{r}\Big)\Big)\Big(y-\Big(r-\frac{1}{r}\Big)\Big) = 0.
\end{align*}
Clearly, we have $-r-\frac{1}{r}<0$ when $r>1$.
Since we are looking for $y\in(0,2)$, the only potential solution is 
\begin{equation}\label{eq:jb}
y=r-\frac{1}{r}.
\end{equation}
Since $r>1$, we have $0<y$.
Moreover, if $r>1$ and $y<2$ in \eqref{eq:jb}, then $r<1+\sqrt{2}$.
Thus, we have $D(r,y)=B(r,y)$ for $(r,y)^\top\in\Omega$ if and only if $(r,y)^\top$ is an element of the graph of the function $f_{DB}: (1,1+\sqrt{2})\to\R$ given by $f_{DB}(r)=r-\frac{1}{r}$.
Similarly, we obtain $D(r,y) < B(r,y)$ if $y<f_{DB}(r)$ or $r\geq 1+\sqrt{2}$, and $D(r,y) > B(r,y)$ if $y>f_{DB}(r)$.

Finally, for $r\in(1,\sqrt{2}]$ we consider
\begin{align*}
&J(r,y)=B(r,y)\\
\Leftrightarrow\quad & r^4 + r^2y^2+1+2r^3y-2r^2-2ry=r^6-2r^4+r^2+r^4y^2 \\ 
\Leftrightarrow\quad & (r^2-1)(r^2y^2-2ry+(r^2-1)^2)=0 \\
\Leftrightarrow\quad & r^2y^2-2ry +(r^2-1)^2= 0 \\
\Leftrightarrow\quad & \Big(y-\Big(\frac{1}{r}-\sqrt{2-r^2}\Big)\Big)\Big(y-\Big(\frac{1}{r}+\sqrt{2-r^2}\Big)\Big) = 0.
\end{align*}
Moreover, if $r\in(1,\sqrt{2}]$, we have $1>\frac{1}{r}>\sqrt{2-r^2}$, i.e., $(r,\frac{1}{r}-\sqrt{2-r^2})^\top,(r,\frac{1}{r}+\sqrt{2-r^2})^\top\in \Omega$.
Thus, $J(r,y)=B(r,y)$ for $(r,y)^\top\in\Omega$ if and only if  $(r,y)^\top$ is an element of the graph of one of the functions $f_{JB,1},f_{JB,2}:(1,\sqrt{2}]\to\R$ given by $f_{JB,1}(r)=\frac{1}{r}-\sqrt{2-r^2}$ and $f_{JB,2}(r)=\frac{1}{r}+\sqrt{2-r^2}$.

Similarly, we obtain $B(r,y) < J(r,y)$ for $r\in(1,\sqrt{2})$ and $y\in (f_{JB,1}(r),f_{JB,2}(r))$ on the one hand, and $B(r,y) > J(r,y)$ for $r\in(1,\sqrt{2}]$ and $y\notin [f_{JB,1}(r),f_{JB,2}(r)]$, and for $r>\sqrt{2}$ on the other hand.

Observe that $f_{JB,1}(\sqrt{2})=f_{DB}(\sqrt{2})=f_{JD}(\sqrt{2})=f_{JB,2}(\sqrt{2})=\frac{1}{\sqrt{2}}$ motivates to look at the cases $r<\sqrt{2}$ and $r>\sqrt{2}$ separately.
As previously mentioned, we have $B(r,y)>J(r,y)$ for $(r,y)^\top\in\Omega$ with $r>\sqrt{2}$, or $r=\sqrt{2}$ and $y\neq \frac{1}{\sqrt{2}}$.
This means that the minimum among $B(r,y)$, $J(r,y)$, and $D(r,y)$ is one of the latter two numbers in these cases.
Hence, for $r\geq \sqrt{2}$, only $f_{JD}$ is relevant for partioning $\Omega$ according to the minimum of $B$, $D$, and $J$.

On the other hand, if $1<r<\sqrt{2}$, then $f_{JB,1}(r)<f_{DB}(r)<f_{JD}(r)<f_{JB,2}(r)$.
Consequently, if $y<f_{DB}(r)$, then $y<f_{JD}(r)$, i.e., $D(r,y)<B(r,y)$ and $D(r,y)<J(r,y)$.
This means that if $(r,y)^\top\in\Omega$ is such that $1<r<\sqrt{2}$ and $y<f_{DB}(r)$, the minimum among $B(r,y)$, $J(r,y)$, and $D(r,y)$ is $D(r,y)$.
Similarly, one finds that the minimum is $B(r,y)$ when $f_{DB}(r)<y<f_{JB,2}(r)$, and $J(r,y)$ for $y>f_{JB,2}(r)$.

Putting everything together, let us define the functions $\widetilde{\Psi}_1,\widetilde{\Psi}_2:(1,\infty)\to(0,2)$ given by 
\begin{align*}
\widetilde{\Psi}_1(r)&= 
\begin{cases}
r-\frac{1}{r} & \text{if } r\in\big(1,\sqrt{2}\big),\\
\frac{-r^2+\sqrt{r^4+6r^2-7}+1}{2r} & \text{if } r\geq\sqrt{2}
\end{cases}\\\intertext{and}
\widetilde{\Psi}_2(r)&=
\begin{cases}
\frac{1}{r} + \sqrt{2-r^2} & \text{if } r\in\big(1,\sqrt{2}\big),\\
\frac{-r^2+\sqrt{r^4+6r^2-7}+1}{2r} & \text{if } r\geq\sqrt{2}.
\end{cases}
\end{align*}
If we set $\widetilde{\Dcal}\defeq\setcond{(r,y)^\top\in\R^2}{1<r,\;0<y\leq\widetilde{\Psi}_1(r)}$, $\widetilde{\Jcal}\defeq\setcond{(r,y)^\top\in\R^2}{1<r,\;2>y\geq\widetilde{\Psi}_2(r)}$, and $\widetilde{\Bcal}\defeq\Omega\setminus\big(\widetilde{\Jcal}\cup\widetilde{\Dcal}\big)$, we have that $D$ is optimal in $\widetilde{\Dcal}$, $J$ is optimal in $\widetilde{\Jcal}$, and $B$ is optimal in $\widetilde{\Bcal}$.

Returning to the original coordinates $r$ and $\theta$, we just need to define the functions $\Psi_1,\Psi_2:(1,\infty)\to (\frac{\uppi}{2},\uppi)$ branch-wise by 
$\Psi_1(r) = \arccos(-\frac{1}{2}\widetilde{\Psi}_1(r))$ and $\Psi_2(r) = \arccos(-\frac{1}{2}\widetilde{\Psi}_2(r))$, where we choose the principal branch of the $\arccos$ function.
The domains $\Dcal, \Jcal, \Bcal$ can be obtained in the obvious way.
Note that inside the definition of these domains, the signs of the orderings involving $y$ are not perturbed by this transformation.
This is due to $\arccos$ being decreasing on its principal branch.
For example, the condition $y<r-\frac{1}{r}$ simply becomes $\theta<\arccos\Big(-\frac{1}{2}\big(r-\frac{1}{r}\big)\Big)$.
\end{proof}

\begin{figure}[h!]
    \centering
\begin{tikzpicture}
\begin{axis}[height=3cm,scale only axis,ymin=pi/2,ymax=pi,xmin=1,xmax=3,ytick={0.5*pi,0.625*pi,0.75*pi,0.875*pi,pi},yticklabels={$\frac{\uppi}{2}$,$\frac{5\uppi}{8}$,$\frac{3\uppi}{4}$,$\frac{7\uppi}{8}$,$\uppi$},ylabel={$\theta$},xlabel={$r$},xmajorgrids={true},ymajorgrids={true},label style={font=\tiny},tick label style={font=\tiny},xticklabel style={text height=2em,yshift=1em}]
\addplot[line width=1pt,domain=1:sqrt(2),samples=200] ({x},{rad(acos(-0.5*(x-1/x)))}); 
\addplot[line width=1pt,domain=1:sqrt(2),samples=200] ({x},{rad(acos(-0.5*(sqrt(2-x^2)+1/x)))}); 
\addplot[line width=1pt,domain=sqrt(2):3,samples=200] ({x},{rad(acos(-0.5*((-x^2 + sqrt(x^4 + 6*x^2 - 7) + 1)/(2*x))))});
\addplot[line width=1pt,dashed,domain=1.93216:pi,samples=200,trig format=rad] ({-cos(x)+sqrt(cos(x)^2+1)},{x});
\draw (axis cs:2.5,2.5) node{$\Jcal$};
\draw (axis cs:2,1.7) node{$\Dcal$};
\draw (axis cs:1.2,2) node{$\Bcal$};
\end{axis}
\end{tikzpicture}
\quad
\begin{tikzpicture}
\begin{axis}[height=3cm,scale only axis,xmin=pi/2,xmax=pi,ymin=pi/2,ymax=pi,xtick={0.5*pi,0.625*pi,0.75*pi,0.875*pi,pi},xticklabels={$\frac{\uppi}{2}$,$\frac{5\uppi}{8}$,$\frac{3\uppi}{4}$,$\frac{7\uppi}{8}$,$\uppi$},ytick={0.5*pi,0.625*pi,0.75*pi,0.875*pi,pi},yticklabels={$\frac{\uppi}{2}$,$\frac{5\uppi}{8}$,$\frac{3\uppi}{4}$,$\frac{7\uppi}{8}$,$\uppi$},ylabel={$\theta$},xlabel={$\delta$},xmajorgrids={true},ymajorgrids={true},label style={font=\tiny},tick label style={font=\tiny}]
\addplot[line width=1pt,domain=pi/2:1.93216,samples=200] ({x},{x});
\addplot[line width=1pt,dashed,domain=1.93216:pi,samples=200] ({x},{x});
\addplot[line width=1pt,mark=none] table [x=theta, y=delta, col sep=space] {data.txt};
\addplot[line width=1pt,mark=none] table [y=theta, x=delta, col sep=space] {data.txt};
\draw (axis cs:2.5,2.2) node{$\Jcal$};
\draw (axis cs:2,1.7) node{$\Dcal$};
\draw (axis cs:1.7,2) node{$\Bcal$};
\end{axis}
\end{tikzpicture}
\quad
\begin{tikzpicture}
\begin{axis}[height=3cm,scale only axis,xmin=0,xmax=1,ymin=0,ymax=1,ylabel={$y$},xlabel={$x$},xmajorgrids={true},ymajorgrids={true},label style={font=\tiny},tick label style={font=\tiny},xticklabel style={text height=2em,yshift=1em}]
\addplot[line width=1pt,domain=0:90,samples=200] ({cos(x)},{sin(x)}); 
\addplot[line width=1pt,domain=135:152.11443,samples=200] ({-sqrt(2)*cos(x)-1},{sqrt(2)*sin(x)}); 
\addplot[line width=1pt,dashed,domain=152.11443:180,samples=200] ({-sqrt(2)*cos(x)-1},{sqrt(2)*sin(x)}); 
\addplot[line width=1pt,mark=none] table [y=yval, x=xval, col sep=space] {data2.txt};
\addplot[line width=1pt,mark=none] table [y=yval, x=xval, col sep=space] {data3.txt};
\draw (axis cs:0.6,0.2) node{$\Jcal$};
\draw (axis cs:0.5,0.7) node{$\Dcal$};
\draw (axis cs:0.1,0.6) node{$\Bcal$};
\end{axis}
\end{tikzpicture}

    \caption{We parametrize \para{} parallelograms $K$ in three different ways by the larger interior angle $\theta\in (\frac{\uppi}{2},\uppi)$, the side length ratio $r\in (1,\infty)$, the larger angle $\delta\in (\frac{\uppi}{2},\uppi)$ formed by the diagonals, and the coordinates $x,y\in\R$.
    (In the latter case, we assume that the vertices of $K$ are $\pm (1,0)$ and $\pm (x,y)$, where $x^2+y^2<1$ and $x,y > 0$.) The solid lines separate regions which correspond to the reflection axis at which the Minkowski chirality $\alpha_1$ is attained: a bisector of an interior angle for $\Bcal$, an angle bisector of the diagonals for $\Dcal$, and a principal axis of the John ellipse for $\Jcal$.
    The dashed line indicates the parallelograms with $\alpha_1(K) = \sqrt{2}$.
    \label{fig:phase-diagram-parallelogram}}
\end{figure}
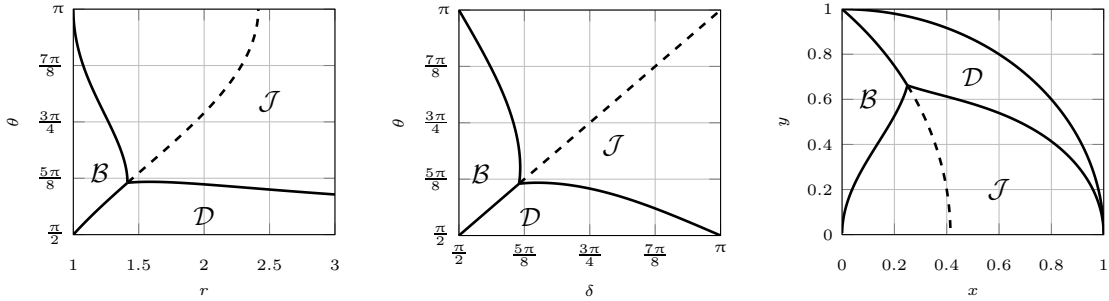

Next, we comment on how to derive phase diagrams for two different parameterizations of parallelograms.
In particular, we identify the pairs of parameters that correspond to parallelograms $K$ for which two of our three candidate reflection axes $U$ attain $\alpha_1(K)$.

In \cref{rem:polarity}, we saw that the diagram in \cref{fig:phase-diagram-parallelogram} becomes symmetric if in the parameterization of \para{} parallelograms we replace the side length ration $r$ by the angle $\delta\in (\frac{\uppi}{2},\uppi)$ formed by the intersection of the diagonals.
Then the pairs $(\delta,\theta)$ that represent \para{} parallelograms $K$ for which the reflections across the bisector of the interior angle $\theta$ and the bisector of the diagonal angle $\delta$ give the same circumradius $R(K,\Phi_U(K))$ are given by $\delta=\theta$.
Among those, there are also the \para{} parallelograms $K$ with $\alpha_1(K)=\sqrt{2}$.

In \cref{thm:parallelogram-phase-diagram}, we saw that the \para{} parallelograms $K$ for which the bisector of the interior angle $\theta$ and the principal axes of $\John(K)$ give the same value for $R(K,\Phi_U(K))$ are given by the side length ratio $r\in (1,\sqrt{2})$ and the identity $-2\cos(\theta)=\frac{1}{r}+\sqrt{2-r^2}$.
For such a parallelogram, the squared lengths of the diagonals can be computed as $r^2+1-2r\cos(\theta)=r^2+2+r\sqrt{2-r^2}$ and $r^2+1+2r\cos(\theta)=r^2-r\sqrt{2-r^2}$ due to the law of cosines, see again \cref{fig:diagonalbisector}.
Hence, another application of the law of cosines yields the following relation between the side length ratio $r$ and the obtuse angle $\delta$ formed by the diagonals of the parallelogram:
\begin{equation*}
    r^2=\frac{r^2+1}{2}-\frac{\sqrt{(r^2+2+r\sqrt{2-r^2})(r^2-r\sqrt{2-r^2})}}{2}\cdot \cos(\delta).
\end{equation*}
This is equivalent to
\begin{equation*}
 \delta =\arccos\left(\frac{1-r^2}{\sqrt{(r^2+2+r\sqrt{2-r^2})(r^2-r\sqrt{2-r^2})}}\right).
\end{equation*}
In the parametrization by $\delta$ and $\theta$, the \para{} parallelograms $K$ for which the bisector of the interior angle $\theta$ and the principal axes of $\John(K)$ give the same value for $R(K,\Phi_U(K))$ are thus represented by the parameterized curve
\begin{equation*}
    (\delta,\theta)=\left(\arccos\left( \frac{1-r^2}{\sqrt{(r^2+2+r\sqrt{2-r^2})(r^2-r\sqrt{2-r^2})}}\right),\arccos\left(-\frac{1}{2}\left(\frac{1}{r}+\sqrt{2-r^2}\right)\right)\right),
\end{equation*}
where $r\in (1,\sqrt{2})$.
Using the symmetry of this parameterization, we also get the pairs $(\delta,\theta)$ representing \para{} parallelograms $K$ for which the bisector of the angle $\delta$ formed by the diagonals and the principal axes of $\John(K)$ give the same value for $R(K,\Phi_U(K))$, see the middle panel of \cref{fig:phase-diagram-parallelogram}.

As a third and last parametrization, let us consider \para{} parallelograms with vertices $\pm (1,0)^\top$ and $\pm (x,y)^\top$, where $x^2+y^2<1$ and $x,y>0$.
This family of parallelograms consist of one representative of each similarity class, each point-symmetric with respect to the coordinate origin and having the longer diagonal of length $2$.
The \para{} parallelograms $K$ for which the reflections across the bisector of the interior angle $\theta$ and the bisector of the diagonal angle $\delta$ give the same circumradius $R(K,\Phi_U(K))$ are those for which the side length ratio equals the diagonal length ratio.
In this parameterization, 
\begin{equation*}
    \frac{4}{4(x^2+y^2)}=r^2=\frac{(x+1)^2+y^2}{(x-1)^2+y^2},
\end{equation*}
i.e., $(x+1)^2+y^2=2$.
Among those, there are also the \para{} parallelograms $K$ with $\alpha_1(K)=\sqrt{2}$.
It remains to locate the parallelograms $K$ for which $\alpha_1(K)$ is obtained by the reflection across the principal axis of $\John(K)$ and one of the other two candidates.
The pairs $(x,y)$ corresponding to \para{} parallelograms with side length ratio $r>1$ are given by
\begin{equation*}
    r=\frac{\sqrt{(x+1)^2+y^2}}{\sqrt{(x-1)^2+y^2}}.
\end{equation*}
Squaring and rearranging yields the equation of the circle $C_1(r)$ with midpoint $(\frac{r^2+1}{r^2-1},0)^\top$ and radius $\sqrt{\frac{(r^2+1)^2}{(r^2-1)^2}-1}$.
Similarly, the pairs $(x,y)$ corresponding to \para{} parallelograms with obtuse interior angle $\theta$ belong to the circle $C_2(\theta)$ with midpoint  $\left(0,-\sqrt{\frac{1}{\sin(\theta)^2}-1}\right)^\top$ and radius $\frac{1}{\sin(\theta)}$.
For every $r>1$ and $\theta\in (\frac{\uppi}{2},\uppi)$, the circles $C_1(r)$ and $C_2(\theta)$ have two intersection points $(x,y)^\top$.
Precisely one of them, call it $v(r,\theta)$, has $x,y>0$.
The common boundary of the domains $\mathcal{J}$ and $\mathcal{B}\cup \mathcal{D}$ in the right panel of \cref{fig:phase-diagram-parallelogram} is then obtained as the set of the points $v(r,\theta)$ for $r>1$ and $\theta=\Psi_2(r)$, with $\Psi_2$ from \cref{thm:parallelogram-phase-diagram}.
Alternatively, by taking \labelcref{eq:alpha1-parallelogram,eq:john-lambda} into account, one arrives at
\begin{equation*}
\alpha_1(K)=\min\setn{\frac{\sqrt{(x+1)^2+y^2}}{\sqrt{(x-1)^2+y^2}},\frac{1}{\sqrt{x^2+y^2}},\frac{1-x^2-y^2 +2x}{\sqrt{(1+x^2-y^2)^2+(2xy)^2}}},
\end{equation*}
for which one could do an analysis analogous to that of \cref{thm:parallelogram-phase-diagram}.
 
\section{Triangles}\label{chap:triangle}

In this section, we prove \cref{thm:triangle}.
By \cref{thm:dbm}, any triangle $K \in \CK^2$ satisfies $\alpha_1(K)\geq 1$, with equality precisely if $K$ is an isosceles triangle.
Therefore, it remains to prove \cref{thm:triangle} for scalene triangles $K$, i.e., those in which all three sides are of different lengths.
To this end, we prove in \cref{thm:under-rotation,thm:triangle-candidates} that if $\alpha_1(K)=R(K,\Phi_U(K))$ for some straight line $U$, then $U$ is parallel to the bisector of one of the interior angles or perpendicular to one of the edges.
In \cref{thm:dilation-factors}, we find the circumradii of triangles with respect to their mirror images upon reflection across straight lines parallel to angle bisectors or perpendicular to edges.
In \cref{thm:triangle-rule-out}, we reduce the list of candidates for the optimal reflection axis to the bisectors of the largest and smallest interior angles and the perpendicular lines to the longest edge, and give an explicit formula for $\alpha_1(K)$ in terms of the side lengths of the triangle.
Last, we describe a \enquote{phase diagram} for the optimal reflection axes in \cref{thm:triangle-phases}, i.e., we provide parametrizations of sets of representatives of the similarity classes of triangles and establish explicit conditions on the parameters which tell us which of the candidates for the optimal reflection axis is the one that attains $\alpha_1$.

As used in the previous section for parallelograms, one may for $U,U' \in \gr{1}{\R^2}$ understand the mirror image $\Phi_{U'}(K)$ as the image of $\Phi_U(K)$ under a rotation.
We are therefore interested in finding the largest dilate of a rotation of $\Phi_{U}(K)$ that fits in $K$.
Formally, we say for $K,L \in \CK^2$ with $K \subset L$ that $K$ is \cemph{optimally contained} in $L$ \cemph{under rotation} if for all rotations $\Psi:\R^2\to\R^2$, $\lambda > 1$, and $x \in \R^2$, we have $\lambda \Psi(K) + x \not\subset L$.

\begin{proposition}\label{thm:under-rotation}
Let $S,T \in \CK^2$ be triangles such that $S$ is optimally contained in $T$ under rotation.
Then at least one edge of $S$ is contained in an edge of $T$.
\end{proposition}
\begin{proof}
    Denote by $u$, $v$, $w$ the vertices of $S$, and by $a$, $b$, $c$ the vertices of $T$.
    Since $S$ is optimally contained in $T$ under rotation, we have $R(S,T) = 1$. Based on the observation following \cref{thm:brandenberg-koenig}, we distinguish three cases.\\[\baselineskip]
    \emph{Case 1: $S$ and $T$ share at least two vertices.}
     Since $S$ and $T$ have a common edge, we are done.  

    \emph{Case 2: $S$ and $T$ share precisely one vertex, say $u$.}
    We may assume that neither of the remaining vertices of $S$ is contained in an edge of $T$ incident to $u$.
    Since $R(S,T) = 1$, at least one vertex of $S$, say $v$, must be an element of the relative interior of the edge of $T$ opposite $u$.
    We are done if $w$ is also an element of this edge, so assume for a contradiction that $w \in \inte(T)$.
    Then
    there is a rotation $\Psi$ around $u$ (with a sufficiently small rotation angle and an appropriately chosen rotation direction) such that $\Psi(v)$ and $\Psi(w)$ both lie in $\inte(T)$.
    This means that $\Psi(S)\subset T$ and $\Psi(S)$ intersects only two edges of $T$.
    By \cref{thm:brandenberg-koenig}, $\Psi(S)$ is not optimally contained in $T$, so $S$ is not optimally contained in $T$ under rotation.
    This is the desired contradiction.
    \\[\baselineskip]
    \emph{Case 3: $S$ and $T$ do not share any vertices.}
    By $R(S,T) = 1$, the vertices $u$, $v$, $w$ of $S$ are elements of the relative interiors of the edges $E_u$, $E_v$, $E_w$ of $T$, respectively.
    Denote the straight lines supporting $T$ at $E_u=[a,b]$, $E_v=[b,c]$, $E_w=[c,a]$ by $F_u$, $F_v$, $F_w$, respectively.
Let
\begin{equation*}
    \Psi_\varphi:\R^2\to\R^2,\qquad \Psi_\varphi(x)=u+\begin{pmatrix}\cos(\varphi)&-\sin(\varphi)\\\sin(\varphi)&\phantom{-}\cos(\varphi)\end{pmatrix}(x-u)
\end{equation*}
be the rotation around $u$ by the angle $\varphi$.
As in the previous case, there exists $\eps > 0$ such that the rotations of $v$ and $w$ around $u$ by an angle at most $\varepsilon$ and with respective appropriately chosen rotation directions lie in $\inte(T)$.
We may assume that the appropriate rotation directions for $v$ and $w$ do no coincide, as otherwise some rotation of $S$ is contained in $T$ and intersects only one edge of $T$,
contradicting the assumption that $S$ is optimally contained in $T$ under rotation.
Thus, we can relabel $v$ and $w$ if necessary such that
$\Psi_\varphi(v) \notin T$ and $\Psi_\varphi(w) \in \inte(T)$ for all $\varphi \in (0,\eps)$, and $R_\varphi(v) \in \inte(T)$ and $R_\varphi(w) \notin T$ for all $\varphi \in (-\eps,0)$.
Our goal is to show that for some $\varphi\in (-\eps,\eps)\setminus\setn{0}$, the triangle $\Psi_\varphi(S)$ can be translated parallel to $F_u$ such that the translated triangle lies entirely in $T$ and intersects at most two edges of $T$, see \cref{fig:case3} for an illustration.
By \cref{thm:brandenberg-koenig}, this shows that $x+\Psi_\varphi(S)$ is not optimally contained in $T$ and, hence, that $S$ is not optimally contained in $T$ under rotation.
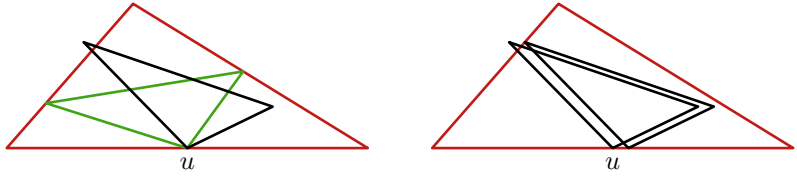
\begin{figure}
    \centering

\begin{tikzpicture}[line cap=round,line join=round,>=stealth,x=1.2cm,y=1.2cm]
\draw[dred,line width=1pt] (0.,0.) -- (1.4,1.6) -- (4.,0.) -- cycle;
\draw[dgreen,line width=1pt] (2.,0.) -- (0.43543079765994286,0.497635197325649) -- (2.619109799087832,0.8497785851767189) -- cycle;
\draw[line width=1pt] (2.,0.) -- (0.8521929627183726,1.1739065481098203) -- (2.9455883872248725,0.4596555113310254) -- cycle;
\draw (2,0) node[below]{$u$};
\end{tikzpicture}
\qquad
\begin{tikzpicture}[line cap=round,line join=round,>=stealth,x=1.2cm,y=1.2cm]
\draw[dred,line width=1pt] (0.,0.) -- (1.4,1.6) -- (4.,0.) -- cycle;
\draw[line width=1pt] (2.,0.) -- (0.8521929627183726,1.1739065481098203) -- (2.9455883872248725,0.4596555113310254) -- cycle;
\draw[line width=1pt] (1.0271682295960927,1.1739065481098203) -- (3.1205636541025927,0.45965551133102545) -- (2.1749752668777202,0.) -- cycle;
\draw (2,0) node[below]{$u$};
\end{tikzpicture}

\caption{Case 3 in the proof of \cref{thm:under-rotation}: $S$ (green) is optimally contained in $T$ (red), but not under rotation.
This is because after suitable rotation (left panel) and translation (right panel), the transformed $S$ is a subset of $T$ but does not intersect all edges of $T$.} \label{fig:case3}
\end{figure}

To this end, let $d_u\defeq a-b$, $d_v\defeq b-c$, and $d_w\defeq c-a$ be direction vectors of $F_u$, $F_v$, and $F_w$, respectively.
The vectors $d_u, d_v, d_w$ are pairwise linearly independent.
Now, note for $z \in \setn{v,w}$ and $x \in \R^2$ that there exists
a unique pair of numbers $(\mu_z(x),\lambda_z(x))$ with $x + \mu_z(x) d_u = z + \lambda_z(x) d_z$.
Note that the map $\mu_z:\R^2\to\R$ given by
\begin{equation*}
	\mu_z(x) = \begin{pmatrix}1&0\end{pmatrix} \left(\begin{array}{c|c}d_u& -d_z\end{array}\right)^{-1} (z-x)
\end{equation*}
is an affine map.
The quantity $\mu_z(x)$ can be understood as a signed distance that measures how far $x$ is from $F_z$ and whether $F_z$ can be reached from $x$ by going in the direction of $d_u$ or its negative.
Our goal is to show $\mu_w(\Psi_\varphi(w)) > \mu_v(\Psi_\varphi(v)) > 0$ or $0 > \mu_w(\Psi_\varphi(w)) > \mu_v(\Psi_\varphi(v))$ for some $\varphi \in (-\eps,\eps) \setminus \setn{0}$.

Next, we define for $\varphi \in (-\frac{\uppi}{2},\frac{\uppi}{2})$ a map $p_\varphi:\R^2\setminus \setn{u}\to\R^2$ by
\begin{equation*}
    p_\varphi(x) \defeq u + \frac{1}{\cos(\varphi)} (\Psi_\varphi(x) - u)= x + \tan(\varphi) (\Psi_{\frac{\uppi}{2}}(x)-u).
\end{equation*}
Then $p_\varphi(x)$ is the unique intersection point of the ray from $u$ through $\Psi_\varphi(x)$ and the straight line supporting $u + \gauge{u-x} \B_2^2$ at $x$, see \cref{fig:notation} for an illustration.

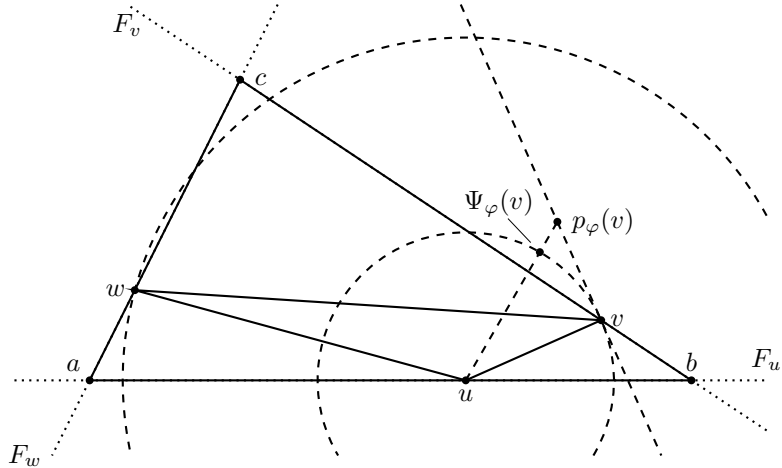
\begin{figure}[ht]
\newcommand\scale{4}
\newcommand\ds{1.5/\scale pt}
\begin{tikzpicture}[scale=\scale]
\draw[thick,dashed] ({(0.355-0.2*1.25)/0.45},1.25) -- ({(0.355+0.2*0.25)/0.45},-0.25); 

\draw[thick,dotted] (-1.25,0) --(1.25,0) node[anchor=south] {$F_u$}; \draw[thick,dotted] (1.25,-0.25/1.5) -- (1-1.5*1.25,1.25) node[anchor=north] {$F_v$}; \draw[thick,dotted] (0.5*1.25-1,1.25) -- (-0.25*0.5-1,-0.25) node[anchor=east] {$F_w$}; 

\draw[thick,dashed] (0.25,0) -- (0.554491,0.527394);

\draw[thick] (0.25,0) -- (0.7,0.2) -- (-0.85,0.3) -- cycle;
\draw[thick] (-1,0) -- (1,0) -- (-0.5,1) -- cycle;

\draw[dashed,thick] ({0.25+(3*sqrt(2)/10)},-0.25) arc({atan2(-0.25,3*sqrt(2)/10))}:{360+atan2(-0.25,-3*sqrt(2)/10))}:{sqrt(0.45*0.45+0.2*0.2)});
\draw[dashed,thick] (1.25,{sqrt(3/10)}) arc({atan2(sqrt(3/10),1)}:{360+atan2(-0.25,-3*sqrt(55)/20}:{sqrt(1.1*1.1+0.3*0.3)});

\fill (0.25,0) circle(\ds) node[anchor=north] {$u$};
\fill (0.7,0.2) circle(\ds) node[anchor=west] {$v$};
\fill (-0.85,0.3) circle(\ds) node[anchor=east] {$w$};
\fill (-1,0) circle(\ds) node[anchor=south east] {$a$};
\fill (1,0) circle(\ds) node[anchor=south] {$b$};
\fill (-0.5,1) circle(\ds) node[right=2pt] {$c$};

\fill (0.496221,0.426468) circle(\ds);
\draw (0.36,0.59) node {$\Psi_\varphi(v)$};
\draw[ultra thin,shorten <=5pt,shorten >=5pt](0.496221,0.426468)--(0.36,0.57);
\fill (0.554491,0.527394) circle(\ds) node[right=2pt] {$p_\varphi(v)$};
\end{tikzpicture}
\caption{Notation used in Case 3 in the proof of \cref{thm:under-rotation}.}
\label{fig:notation}
\end{figure}

Let further
\begin{equation*}
    c_z \defeq \begin{pmatrix}1&0\end{pmatrix} \left(\begin{array}{c|c}d_u& -d_z\end{array}\right)^{-1}  (u-\Psi_{\frac{\uppi}{2}}(z)).
\end{equation*}
Then $\mu_z(p_\varphi(z))= c_z \tan(\varphi)$ for all $\varphi \in (-\frac{\uppi}{2},\frac{\uppi}{2})$.
If $\varphi\in (-\frac{\uppi}{2},\frac{\uppi}{2})\setminus\setn{0}$, then $\frac{1}{\cos(\varphi)}>1$, so $\Psi_\varphi(z) \in (u,p_\varphi(z))$.
Since $\mu_z$ is an affine map, $\mu_z(\Psi_\varphi(z))$ is between $\mu_z(u)$ and $\mu_z(p_\varphi(z))$.
Next, we determine the signs and the order of the latter three numbers, depending on the sign of $\varphi \in (-\eps,\eps)\setminus\setn{0}$.
Since $u + [0,\infty) d_u$ does not intersect $F_v$ but $F_w$, we already know that $\mu_v(u) < 0$ and $\mu_w(u)>0$.

Assume first that $\varphi\in (0,\eps)$.
Then $\Psi_\varphi(v) \notin T$ by assumption and $[u,\Psi_\varphi(v)) \cap F_v\neq \emptyset$ if $\eps$ is sufficiently small.
Since $\mu_v(x) = 0$ precisely if $x \in F_v$, we conclude from $\mu_v(u) < 0$ and the affinity of $\mu_v$ that $0 < \mu_v(\Psi_\varphi(v)) < \mu_v(p_\varphi(v))$.
In particular, we obtain from $\mu_v(p_\varphi(v))= c_v \tan(\varphi)$  that $c_v>0$.
Similarly, $\Psi_\varphi(w)\in\inte(T)$ by assumption and $[u,p_\varphi(w)]\cap F_w =\emptyset$ if $\eps$ is sufficiently small.
Since $\mu_w(x)=0$ precisely if $x\in F_w$, we conclude similar to before that $\mu_w(u)>\mu_w(\Psi_\varphi(w)) > \mu_w(p_\varphi(w)) > 0$ and thus $c_w>0$.

Second, for $\varphi\in (-\eps,0)$, we have $\Psi_\varphi(v) \in \inte(T)$ by assumption and $[u,p_\varphi(v)]\cap F_v=\emptyset$ if $\eps$ is small enough.
Since $\mu_v(x) = 0$ precisely if $x \in F_v$, we conclude that $\mu_v(\Psi_\varphi(v)) < \mu_v(p_\varphi(v)) < 0$.
Similarly, $0 > \mu_w(\Psi_\varphi(w)) > \mu_w(p_\varphi(w))$ if $\eps$ is small enough.

Finally, we distinguish two cases.
If $c_w \geq c_v$, then we have for any  $\varphi\in (0,\eps)$ that
\begin{equation}\label{eq:compare-coeff}
	\mu_w(\Psi_\varphi(w))
	> \mu_w(p_\varphi(w))
	= c_w \tan(\varphi)
	\geq c_v \tan(\varphi)
	= \mu_v(p_\varphi(v))
	> \mu_v(\Psi_\varphi(v))
	> 0.
\end{equation}
By definition of $\mu_w$, we always have $\Psi_\varphi(w) + \mu_w(\Psi_\varphi(w)) d_u \in F_w$.
Since $\Psi_\varphi(w)\in\inte(T)$, we even obtain for $\eps$ sufficiently small that $\Psi_\varphi(w) + \mu_w(\Psi_\varphi(w)) d_u$ lies in the relative interior of $E_w$.
Further, we have by \eqref{eq:compare-coeff} that
\begin{equation*}
	\Psi_\varphi(w) + \mu_v(\Psi_\varphi(v)) d_u
	\in [\Psi_\varphi(w), \Psi_\varphi(w) + \mu_w(\Psi_\varphi(w)) d_u)
	\subset \inte(T).
\end{equation*}
Moreover, we also have $\Psi_\varphi(v) + \mu_v(\Psi_\varphi(v)) d_u\in E_v$.
By $\mu_w(u)>\mu_w(\Psi_\varphi(w)) > \mu_v(\Psi_\varphi(v))$, we further obtain $\Psi_\varphi(u) + \mu_v(\Psi_\varphi(v)) d_u=u + \mu_v(\Psi_\varphi(v)) d_u\in E_u \setminus E_w$.
Altogether, $\Psi_\varphi(S) + \mu_v(\Psi_\varphi(v)) d_u \subset T$ and $\Psi_\varphi(S) + \mu_v(\Psi_\varphi(v)) d_u$ does not intersect $F_w$.
This completes the proof in this case.

If instead $c_v \geq c_w$, we have for any $\varphi \in (-\eps,0)$ that
\begin{equation*}
	0
	> \mu_w(\Psi_\varphi(w))
	> \mu_w(p_\varphi(w))
	= c_w \tan(\varphi)
	\geq c_v \tan(\varphi)
	= \mu_v(p_\varphi(v))
	> \mu_v(\Psi_\varphi(v)).
\end{equation*}
Analogously, we get $\Psi_\varphi(S) + \mu_w(\Psi_\varphi(w)) d_u \subset T$
and that $\Psi_\varphi(S) + \mu_w(\Psi_\varphi(w)) d_u$ does not intersect $F_v$ for $\eps$ small enough, which completes the proof also in this case.
\end{proof}

We immediately obtain a short list of candidates for the optimal reflection line for the $1$st Minkowski chirality of triangles.

\begin{proposition}\label{thm:triangle-candidates}
Let $K \in \CK^2$ be a triangle and $U \in \gr{1}{\R^2}$ such that $R(K,\Phi_U(K)) = \alpha_1(K)$.
Then $U$ is parallel to the bisector of an interior angle of $K$ or perpendicular to an edge of $K$.
\end{proposition}
\begin{proof}
The triangle $K$ is optimally contained in $x + \alpha_1(K) \Phi_U(K)$ under rotation for some $x \in \R^2$.
The previous proposition therefore shows that at least one edge $E$ of $K$ is a subset of an edge of $x + \alpha_1(K) \Phi_U(K)$.
If $V$ denotes the straight line supporting $K$ at $E$, we see that $\Phi_U(V)$ must be parallel to $V$ or to the affine hull of another edge of $K$.
This means that $U$ is perpendicular to an edge of $K$, parallel to the bisector of an interior angle of $K$, or parallel to an edge of $K$.
The latter case cannot happen, as $x$ would have to be such that $E\subset x + \alpha_1(K) \Phi_U(E)$.
But then the vertex of $K$ opposite $E$ is not an element of $x + \alpha_1(K) \Phi_U(K)$.
\end{proof}

Next, we determine the the minimal dilation factor in the inclusion $K\subset x+\lambda K^\prime$, where $K^\prime$ is the image of $K$ under reflection across a straight line parallel to one of the angle bisectors or perpendicular to one of the edges.

\begin{proposition}\label{thm:dilation-factors}
Let $K=\conv\setn{a,b,c}\in\CK^2$ be a triangle with $x\defeq \gauge{b-c}\leq y\defeq\gauge{a-c}$.
\begin{enumerate}[label={(\roman*)},leftmargin=*,align=left,noitemsep]
\item{Let $U \subset \R^2$ be the bisector of the interior angle at $c$.
Then $R(K,\Phi_U(K))=\frac{y}{x}$.\label{bisector}}
\item{Let $U$ be the perpendicular bisector of $[a,b]$, and $z\defeq\gauge{a-b}$.
Then $R(K,\Phi_U(K))=1+\frac{y^2-x^2}{z^2}$.\label{perpendicular}}
\end{enumerate}
\end{proposition}
\begin{proof}
For \ref{bisector}, note that $\Phi_U$ maps $\aff \setn{a,c}$ onto $\aff \setn{b,c}$ and vice versa.
Thus, $K\subset c+\frac{y}{x} (\Phi_U(K)-c)$, and the containment is optimal by \cref{thm:brandenberg-koenig}, see \cref{fig:triangle-bisector-reflection}.
This means that $R(K,\Phi_U(K))=\frac{y}{x}$.\\[\baselineskip]
For \ref{perpendicular}, note that $\Phi_U(K)=\conv\setn{a,b,\Phi_U(c)}$, and the line segments $[a,c]$ and $[b,\Phi_U(c)]$ have a unique intersection point $d$.
Further, we have
\begin{equation*}
    K\subset a+\frac{y}{\gauge{d-a}}(\Phi_U(K)-a),
\end{equation*}
see \cref{fig:triangle-bisector-reflection}. \cref{thm:brandenberg-koenig} shows the optimality of this containment, i.e., $R(K,\Phi_U(K))=\frac{y}{\gauge{d-a}}$.
Since $\aff\setn{a,b}$ and $\aff\setn{c,\Phi_U(c)}$ are parallel, we get as a consequence of the intercept theorem that
\begin{equation*}
    \frac{\gauge{d-c}}{\gauge{d-a}} =\frac{\gauge{c-\Phi_U(c)}}{z}.
\end{equation*}
Ptolemy's theorem applied to the isosceles trapezoid $\conv\setn{a,b,c,\Phi_U(c)}$ shows $\gauge{c-\Phi_U(c)}=\frac{y^2-x^2}{z}$.
Thus, 
\begin{equation*}
    \frac{y}{\gauge{d-a}}=\frac{\gauge{d-a}+\gauge{d-c}}{\gauge{d-a}}=1+\frac{\gauge{d-c}}{\gauge{d-a}}=1+\frac{\gauge{c-\Phi_U(c)}}{z}=1+\frac{y^2-x^2}{z^2}.\qedhere
\end{equation*}
\end{proof}

\begin{figure}[h]
\newcommand\scale{1.1}
\newcommand\ds{1.5/\scale pt}
\begin{minipage}{0.4\linewidth}
 \centering
     \begin{tikzpicture}[line cap=round,line join=round,>=triangle 45,x=1.0cm,y=1.0cm, scale=\scale]
\pgfmathsetmacro{\x}{1.5}
\coordinate (a) at (30:2);
\coordinate (b) at (-30:1.2);
\coordinate (c) at (0,0);
\coordinate (aprime) at (-30:2);
\coordinate (bprime) at (30:1.2);
\draw[thick] (aprime)--(-30:3.3333)--(a);

\fill[white] (-30:2.4) circle(4pt);

\draw ($(bprime)+(-0.3,0.3)$) node{$\Phi_U(b)$};

\draw[thick, gray] (3,0)--(-1,0) node[left]{$U$};
\draw[thick, dgreen] (a)--(b)--(c)--cycle;
\draw[thick, dred] (aprime)--(bprime)--(c)--cycle;
\fill (a) circle (\ds) node[right]{$a$};
\fill (b) circle (\ds) node[below]{$b$};
\fill (c) circle (\ds) node[above left]{$c$};
\fill (aprime) circle (\ds) node[below]{$\Phi_U(a)$};
\fill (bprime) circle (\ds);
\end{tikzpicture}
\end{minipage}
   \begin{minipage}{0.5\linewidth}
    \centering

\begin{tikzpicture}[line cap=round,line join=round,>=triangle 45,x=1.0cm,y=1.0cm, scale=\scale]
\pgfmathsetmacro{\x}{1.5}
\coordinate (a) at (0,0);
\coordinate (b) at (1,0);
\coordinate (c) at (\x,0.7);
\coordinate (cprime) at (1-\x,0.7);

\draw[dgreen, thick] (a)--(b)--(c)--cycle;
\draw[dred, thick] (a)--(b)--(cprime)--cycle;
\path[name path=line 1] (a) -- (c);
\path [name path=line 2] (b) -- (cprime);
\path [name intersections={of=line 1 and line 2,by=E}];
\draw[gray, thick] (0.5,-0.7)--(0.5,2) node[left]{$U$};
\fill[white] ($(E)+(0,0.25)$) circle(5pt);
\fill (a) circle(\ds) node[below]{$a$};
\fill (b) circle(\ds) node[below]{$b$};
\fill (c) circle(\ds) node[below]{$c$};
\fill (cprime) circle(\ds) node[left]{$\Phi_U(c)$};
\fill (E) circle(\ds) node[above]{$d$};

\draw[thick] (b)--(2*\x,0)--($(a)!2*\x!(cprime)$)--(cprime);
\end{tikzpicture}   
\end{minipage} 

  \caption{Optimal containment of a triangle $K$ (green) in the dilated mirror image after reflection across the reflection axis $U$ (gray) for the case of an angle bisector (left panel) or perpendicular bisector on an edge (right panel): $\Phi_U(K)$ (red), and the appropriate translate of $R(K,\Phi_U(K))\Phi_U(K)$ (black).}\label{fig:triangle-bisector-reflection}
\end{figure}
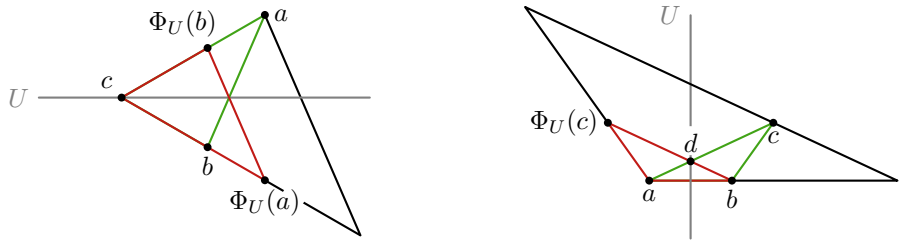

Now we are ready to give a formula for $\alpha_1(K)$ for a triangle $K$ in terms of its side lengths, and rule out three of the candidates from \cref{thm:triangle-candidates}.
\begin{proposition}\label{thm:triangle-rule-out}
Let $K\in\CK^2$ be a triangle with side lengths $0<x\leq y\leq z$.
Then
\begin{equation}\label{eq:best-of-sup}
\alpha_1(K)=\min\setn{\frac{z}{y},\frac{y}{x},1+\frac{y^2-x^2}{z^2}}.
\end{equation}
In particular, we have $\alpha_1(K)=R(K,\Phi_U(K))$ for $U \in \gr{1}{\R^2}$ a line parallel to the bisector of the smallest or the largest interior angle of $K$ or perpendicular to the longest edge of $K$.
\end{proposition}
\begin{proof}
We may assume that $K$ is a scalene triangle, i.e., $x < y < z$, since otherwise all assertions are clear.
From \cref{thm:triangle-candidates,thm:dilation-factors}, we know that
\begin{equation*}
    \alpha_1(K)=\min\setn{\frac{z}{y},\frac{y}{x},\frac{z}{x},1+\frac{y^2-x^2}{z^2},1+\frac{z^2-x^2}{y^2},1+\frac{z^2-y^2}{x^2}}.
\end{equation*}
In order to verify \eqref{eq:best-of-sup}, it suffices to show that
\begin{equation*}
    \min\setn{\frac{z}{y},\frac{y}{x},1+\frac{y^2-x^2}{z^2}}<\min\setn{\frac{z}{x},1+\frac{z^2-x^2}{y^2},1+\frac{z^2-y^2}{x^2}}.
\end{equation*}
First, note that $\frac{y}{x}<\frac{z}{x}$ since $y<z$.
Next, we have
\begin{equation*}
    1+\frac{y^2-x^2}{z^2}<1+\frac{z^2-x^2}{z^2}<1+\frac{z^2-x^2}{y^2}.
\end{equation*}
Last, observe that
    $\frac{z}{y}<1+\frac{z^2-y^2}{x^2}$
is equivalent to $0<x^2y+yz^2-y^3-x^2z=(z-y)(y^2-x^2+yz)$.
The latter is true since we assumed $z>y$ and $y^2>x^2$.
\end{proof}

Note that in the setup of \cref{thm:triangle-rule-out}, we do not always have $1+\frac{y^2-x^2}{z^2}<1+\frac{z^2-y^2}{x^2}$ whenever $0<x<y<z$ and $x+y>z$.
For instance, by taking $x$ sufficiently large, $y=x^2$ and $z=x^2+\frac{1}{4}$ we have
\begin{equation*}
    1+\frac{y^2-x^2}{z^2}=1+\frac{x^4-x^2}{x^4+2x^2+1}\approx 2, \quad \text{while} \quad 1+\frac{z^2-y^2}{x^2}=1+\frac{\frac{1}{2}x^2+\frac{1}{16}}{x^2}\approx \frac{3}{2}.
\end{equation*}
Thus, in the proof of \cref{thm:triangle-rule-out}, we could not have ruled out the perpendicular bisector of the shortest edge in the same way we eliminated the perpendicular bisector of the middle edge.

In the following corollary, we describe the \enquote{phase diagram} depicted in \cref{fig:triangle-phases} for the reflection axis at which $\alpha_1(K)$ of a given scalene triangle $K$ is attained.
We parameterize the triangle in terms of its side lengths $x$, $y$, and $1$, assuming that $0<x<y<1<x+y$.
Later, we give a similar result for triangles parameterized by a vertex $(x,y)^\top$ while fixing the other vertices at $(0,0)^\top$ and $(1,0)^\top$, see \cref{thm:triangle-phases-2,fig:triangle-phases-2}.

\begin{corollary}\label{thm:triangle-phases}
Let $K$ be a triangle with side lengths $x$, $y$, and $1$ such that $0<x<y<1<x+y$.
Define $\Psi_1,\Psi_2:(\frac{1}{2},1)\to\R$ by
\begin{equation*}
    \Psi_1(y)=
    \begin{cases}
        0&\text{if }y\leq \frac{\sqrt{2}}{2},\\
        \sqrt{\frac{y^3+y-1}{y}}&\text{if }\frac{\sqrt{2}}{2}<y\leq y_0,\\
        y^2&\text{if }y>y_0
    \end{cases}
\quad\text{and}\quad
    \Psi_2(y)=
    \begin{cases}
        0&\text{if }y\leq \frac{\sqrt{2}}{2},\\
        \frac{1}{2}(\sqrt{y^2+4}-y)&\text{if }\frac{\sqrt{2}}{2}<y\leq y_0,\\
        y^2&\text{if }y>y_0
    \end{cases}
\end{equation*}
where $y_0\approx 0.819173$ is the unique positive real solution of the equation $y^4+y^3=1$.
Let further
\begin{align*}
    \Omega&\defeq \setcond{(x,y)^\top\in\R^2}{0<x<y<1<x+y}, &\Lcal&\defeq \setcond{(x,y)^\top\in\Omega}{x>\Psi_2(y)},\\
    \Scal&\defeq \setcond{(x,y)^\top\in\Omega}{x\leq \Psi_1(y)}, &\Pcal&\defeq \Omega\setminus\left(\Lcal\cup \Scal\right). \end{align*}
Then $\alpha_1(K)=R(K,\Phi_U(K))$, where
\begin{enumerate}[label={(\roman*)},leftmargin=*,align=left,noitemsep]
\item{$U$ is the bisector of the largest interior angle of $K$ whenever $(x,y)\in \Lcal$, }
\item{$U$ is the bisector of the smallest interior angle of $K$ whenever $(x,y)\in \Scal$, and}
\item{$U$ is perpendicular to the longest edge of $K$ whenever $(x,y)\in \Pcal$.}
\end{enumerate}
\end{corollary}
\begin{proof}
Let $L,S,P : \Omega \to \R$ be defined by
\begin{equation*}
    L(x,y) = \frac{y}{x},
        \quad
    S(x,y) = \frac{1}{y},
        \quad \text{and} \quad
    P(x,y) = 1 + y^2 - x^2.
\end{equation*}
These expressions are by \cref{thm:dilation-factors} equal to $R(K,\Phi_U(K))$ for $U \subset \R^2$ the bisector of the largest interior angle of $K$, the bisector of the smallest interior angle of $K$, or perpendicular to the longest edge of $K$, respectively.
By \cref{thm:triangle-rule-out}, proving the corollary can
be recast as determining the domains in $\Omega$ where each of $L$, $S$, and $P$ is minimal.
As we shall see, this essentially boils down to finding the solutions to the equations $L=S$, $S=P$ and $L=P$.
For $(x,y)^\top\in \Omega$, we have
\begin{align*}
    L(x,y)=S(x,y)&\Leftrightarrow x=y^2,\\
    S(x,y)=P(x,y)&\Leftrightarrow 1=y+y^3-x^2y\Leftrightarrow x=\sqrt{\frac{y^3+y-1}{y}}, \quad \text{and}\\P(x,y)=L(x,y)&\Leftrightarrow y=x+xy^2-x^3 \Leftrightarrow 0=(x-y)(x^2+xy-1)
    \Leftrightarrow x=\frac{1}{2}\left(\sqrt{y^2+4}-y\right).
\end{align*}
There is a unique element $(x_0,y_0)\in\Omega$ for which $S(x_0,y_0)=L(x_0,y_0)=P(x_0,y_0)$ with $y_0\approx 0.819173$ being the unique positive real solution of the equation $y^4+y^3=1$.
Indeed, from $x=y^2$ and $x^2 + xy - 1 = 0$, we obtain $y^4+y^3-1=0$.
The derivative of $y\mapsto y^4+y^3-1$ is positive for $y>0$, so $y\mapsto y^4+y^3-1$ itself is strictly increasing for $y>0$.
Hence, there exists at most one positive real solution $y_0$ of the equation $y^4+y^3-1=0$.
The existence of a solution in $[0,1]$ follows from the intermediate value theorem.

Let $f_{LS} : (\frac{\sqrt{5}-1}{2},1) \to \R$ and $f_{SP}, f_{PL} : (\frac{\sqrt{2}}{2},1) \to \R$ be given by
\begin{equation*}
    f_{LS}(y) = y^2,
        \quad
    f_{SP}(y) = \sqrt{\frac{y^3+y-1}{y}},
        \quad \text{and} \quad
    f_{PL}(y) = \frac{1}{2}\left(\sqrt{y^2+4}-y\right)
\end{equation*}
Note that for $(x,y)^\top\in\Omega$, the statements $x=f_{LS}(y)$ and $L(x,y)=S(x,y)$ are equivalent.
Likewise, $x=f_{SP}(y)$ and $S(x,y)=P(x,y)$ are equivalent, and so are $x=f_{PL}(y)$ and $P(x,y)=L(x,y)$.

The graphs of $f_{LS}$, $f_{SP}$, and $f_{PL}$ subdivide $\Omega$ into six regions, corresponding to the six possible strict orderings of the three numbers $L(x,y)$, $S(x,y)$, and $P(x,y)$, see the left panel of \cref{fig:triangle-phases}.

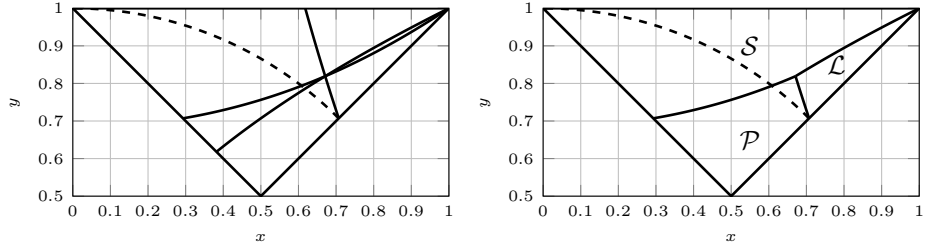
\begin{figure}[ht]
    \centering
\begin{tikzpicture}
\begin{axis}[x=5cm,y=5cm,xtick={0,0.1,...,1},xmajorgrids={true},ytick={0.5,0.6,...,1},ymajorgrids={true},xmin=0,xmax=1,ymin=0.5,ymax=1,xlabel={$x$},ylabel={$y$},label style={font=\tiny},tick label style={font=\tiny}]
\draw [line width=1pt](axis cs:0,1)--(axis cs:0.5,0.5)--(axis cs:1,1)--cycle;
\addplot [line width=1pt,domain=0.618034:1,samples=300] ({x^2},{x});
\addplot [line width=1pt,domain=0.70711:1,samples=300] ({sqrt((x^3+x-1)/x)},{x});
\addplot [line width=1pt,domain=0.618034:0.70711,samples=300] ({x},{1/x-x});
\addplot [dashed,line width=1pt,domain=45:90,samples=300]({cos(x)},{sin(x)});
\end{axis}
\end{tikzpicture}
\begin{tikzpicture}
\begin{axis}[x=5cm,y=5cm,xtick={0,0.1,...,1},xmajorgrids={true},ytick={0.5,0.6,...,1},ymajorgrids={true},xmin=0,xmax=1,ymin=0.5,ymax=1,xlabel={$x$},ylabel={$y$},label style={font=\tiny},tick label style={font=\tiny}]
\draw [line width=1pt] (axis cs:0,1)--(axis cs:0.5,0.5)--(axis cs:1,1)--cycle;
\addplot [line width=1pt,domain=0.8191725:1,samples=300] ({x^2},{x});
\addplot [line width=1pt,domain=0.70711:0.8191725,samples=300] ({sqrt((x^3+x-1)/x)},{x});
\addplot [line width=1pt,domain=0.6710436:0.70711,samples=300] ({x},{1/x-x});
\draw (axis cs:0.55,0.65) node{$\mathcal{P}$};
\draw (axis cs:0.78,0.85) node{$\mathcal{L}$};
\draw (axis cs:0.55,0.9) node{$\mathcal{S}$};
\addplot [dashed,line width=1pt,domain=45:90,samples=300]({cos(x)},{sin(x)});
\end{axis}
\end{tikzpicture}
    \caption{Parameterization of scalene triangles by their side lengths $x$, $y$, and $1$ with $0<x<y<1<x+y$.
    The dashed line indicates the right-angled triangles.
    Left panel: The solid lines inside the parameter space indicate the pairs $(x,y)^\top$ for which two of the three candidate reflection axes lead to the same value of $R(K,\Phi_U(K))$.
    Right panel: The parameter space is separated into regions which correspond to the reflection axis $U$ for which $\alpha_1(K)=R(K,\Phi_U(K))$: the bisector of a largest interior angle for $\mathcal{L}$, the bisector of a smallest interior angle for $\mathcal{S}$, and the perpendicular line to a longest edge for $\mathcal{P}$.
    }
    \label{fig:triangle-phases}
\end{figure}

To determine on which two of the six regions each of the three numbers is smallest, it suffices to pick a point $(x,y)^\top$ from every region and check the ordering for these representatives.
For instance, if $x=0.6$ and $y=0.78$, then $P(x,y)\approx 1.2484<S(x,y)\approx 1.28205<L(x,y)=1.3$.
Thus, if $K$ has side lengths $x$, $y$, and $1$ with $(x,y)^\top \in\Omega$ and in addition $\max\setn{1-y,f_{SP}(y)}<x<f_{LS}(y)$, then $(x,y)^\top\in\Pcal$ and $\alpha_1(K)$ is attained at the reflection axes perpendicular to the longest edge of $K$.
\end{proof}

A different set of representatives of the similarity classes of triangles is given by the triangles of the form
$\conv\setn{(0,0)^\top,(1,0)^\top,(x,y)^\top}$ with $x>\frac{1}{2}$, $y>0$, and $x^2+y^2<1$.
The side lengths of such a triangle are in increasing order given by $\sqrt{(x-1)^2+y^2}$, $\sqrt{x^2+y^2}$, and $1$.
Then, similarly to \cref{thm:triangle-phases}, one can show the following result.
\begin{corollary}\label{thm:triangle-phases-2}
Let $K= \conv\setn{(0,0)^\top,(1,0)^\top,(x,y)^\top}$ where $x>\frac{1}{2}$, $y>0$, and $x^2+y^2<1$.
Define $\Psi_1,\Psi_2:(\frac{1}{2},1)\to\R$ by
\begin{equation*}
    \Psi_1(x)=
    \begin{cases}
        \sqrt{\frac{-2x^2+1+\sqrt{5-8x}}{2}}&\text{if }x\leq x_0,\\
        \frac{\sqrt{1-4x^4}}{2x}&\text{if }x_0<x<\frac{1}{\sqrt{2}},\\
        0&\text{if }x>\frac{1}{\sqrt{2}}
    \end{cases}
\quad\text{and}\quad
    \Psi_2(x)=
    \begin{cases}
    \sqrt{\frac{x^2(3-2x)}{1+2x}}&\text{if }x\leq x_0,\\
    \frac{\sqrt{1-4x^4}}{2x}&\text{if }x_0<x<\frac{1}{\sqrt{2}},\\
    0&\text{if }x>\frac{1}{\sqrt{2}} 
    \end{cases}
\end{equation*}
where $x_0\approx 0.61037$ is the unique positive real solution of the equation $16 x^4-2x-1=0$.
Let further
\begin{align*}
    \Omega&\defeq \setcond{(x,y)^\top\in\R^2}{x>\frac{1}{2},\, y>0,\, x^2+y^2<1},
    &\Scal&\defeq \setcond{(x,y)^\top\in\Omega}{y\geq \Psi_1(x)},\\
    \Pcal&\defeq \setcond{(x,y)^\top\in\Omega}{y<\Psi_2(x)},
    &\Lcal&\defeq \Omega\setminus\left(\Pcal\cup \Scal\right).
\end{align*}
Then $\alpha_1(K)=R(K,\Phi_U(K))$, where
\begin{enumerate}[label={(\roman*)},leftmargin=*,align=left,noitemsep]
\item{$U$ is the bisector of the largest interior angle of $K$ whenever $(x,y)\in \Lcal$,}
\item{$U$ is the bisector of the smallest interior angle of $K$ whenever $(x,y)\in \Scal$, and}
\item{$U$ is perpendicular to the longest edge of $K$ whenever $(x,y)\in \Pcal$.}
\end{enumerate}
\end{corollary}

An illustration of \cref{thm:triangle-phases-2} is given in \cref{fig:triangle-phases-2}.
\begin{figure}[ht]
    \centering
\begin{tikzpicture}
\begin{axis}[x=6cm,y=6cm,xtick={0.5,0.6,...,1},xmajorgrids={true},ytick={0,0.1,...,0.9},ymajorgrids={true},xmin=0.5,xmax=1,ymin=0,ymax=0.9,xlabel={$x$},ylabel={$y$},label style={font=\tiny},tick label style={font=\tiny}]
\addplot [line width=1pt,domain=0.39:0.2,samples=400] ({1-x}, {sqrt(1/(4*(1-2*x+x^2)) - (x-1)^2)})--(axis cs:0.707107,0);
\addplot[line width=1pt,domain=0:60,samples=200] ({cos(x)},{sin(x)}); 
\addplot[line width=1pt,dashed,domain=0:90,samples=200] ({0.5+0.5*cos(x)},{0.5*sin(x)}); 
\addplot [line width=1pt,domain=0.39:0.5,samples=100] ({1-x},{sqrt((((-1 + x)^2 * (1 + 2 * x))/(3 - 2 * x)))});
\addplot [line width=1pt,domain=1:1.5,samples=100] ({1-(1 - x + x^2)/(2 * x^2)}, {0.5 * sqrt((-1 + 2 * x + x^2 + 2 * x^3 - x^4)/(x^4))});
\draw (axis cs:0.6,0.3) node{$\mathcal{P}$};
\draw (axis cs:0.55,0.6) node{$\mathcal{L}$};
\draw (axis cs:0.8,0.2) node{$\mathcal{S}$};
\end{axis}
\end{tikzpicture}  

    \caption{We parameterize triangles $K$ with vertices $(0,0)^\top$ and $(1,0)^\top$ by their third vertex $(x,y)^\top$ with $x\in (\frac{1}{2},1)$ and $y\in (0,\sqrt{1-x^2})$.
    The solid lines separate the regions that correspond to the reflection axis $U$ for which $\alpha_1(K)=R(K,\Phi_U(K))$: the bisector of a largest interior angle for $\mathcal{L}$, the bisector of a smallest interior angle for $\mathcal{S}$, and the perpendicular line to a longest edge for $\mathcal{P}$.
    The dashed line indicates the right-angled triangles.}
    \label{fig:triangle-phases-2}
\end{figure}
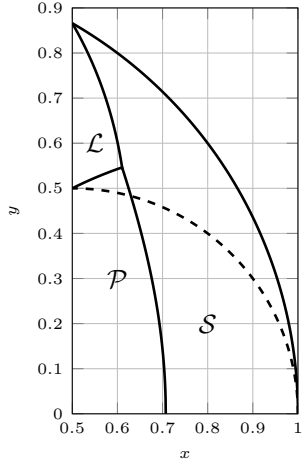

The last part of \cref{thm:triangle} that remains to be proved is \eqref{eq:triangle-range}.

\begin{corollary}
We have $\setcond{\alpha_1(K)}{K\in\CK^2\text{ is a scalene triangle}}=\left(1,\sqrt{2}\right)$.
\end{corollary}

\begin{proof}

Let $0<x \leq y \leq z$ be the side lengths of a triangle $K \in \CK^2$.
Then $z < x+y$ and in particular $x^2 > (z-y)^2$.
By \cref{thm:triangle-rule-out}, we have
\begin{equation*}
    \alpha_1(K)
    \leq 1 + \frac{y^2-x^2}{z^2}
    < 1 + \frac{y^2-(z-y)^2}{z^2}
    = \frac{2y}{z}.
\end{equation*}
Therefore, we directly get $\alpha_1(K) < \sqrt{2}$ whenever $\frac{y}{z} \leq \frac{1}{\sqrt{2}}$.
Otherwise, we use \cref{thm:triangle-rule-out} again to obtain
\begin{equation*}
    \alpha_1(K)
    \leq \frac{z}{y}
    < \sqrt{2}.
\end{equation*}

To complete the proof, by \cref{prop:chirality-attained} and the intermediate value theorem,
it remains to find a sequence $(K_i)_{i\in\N}$ of triangles $K_i\in\CK^2$ such that $\lim_{i\to\infty}\alpha_1(K_i)=\sqrt{2}$.
For $i\in\N$, we choose the triangle $K_i \defeq \conv\setn{ (0,0)^\top, (1,0)^\top, (\frac{1}{\sqrt{2}},\frac{1}{i+2})^\top }$.
Letting $i\to\infty$, the ordered triple of side lengths of $K_i$ converges to $(1-\frac{1}{\sqrt{2}},\frac{1}{\sqrt{2}},1)$.
Therefore, \cref{thm:triangle-rule-out} gives
\begin{equation*}
    \lim_{i\to\infty} \alpha_1(K_i)
    = \min \setn{ \frac{1}{\frac{1}{\sqrt{2}}}, \frac{\frac{1}{\sqrt{2}}}{1-\frac{1}{\sqrt{2}}}, 1 + \frac{\frac{1}{2} - \left(1-\frac{1}{\sqrt{2}} \right)^2}{1} }
    = \sqrt{2}.\qedhere
\end{equation*}
\end{proof}

\section{Open problems}

We close this paper with some final directions for problems to consider in future work.

First, we know from \cref{thm:dbm} for any $j \in \{0, \ldots, n\}$ and $K \in \CK^n$ that $\alpha_j(K) \leq n$.
As outlined in the introduction and \cref{chap:main-theorem}, this bound is tight if and only if it is tight for just simplices.
While \cref{thm:triangle} shows that the latter does not hold for $(n,j) = (2,1)$, the question of tightness remains open in the general case.
As a starting point for deriving better bounds on the Minkowski chiralities, it would therefore be interesting to determine the tightness of $\alpha_j(K) \leq n$ for simplices $K \in \CK^n$.
Once a pair $(n,j)$ for which the tightness fails is found, one can consider strengthening the upper bound $n$.
In particular, finding the best possible upper bound of the $1$st Minkowski chirality of planar convex bodies appears to be an intriguing problem.

Second, \cref{thm:upperboundsymmetric,thm:parallelogram} show that any point-symmetric $K \in \CK^2$ satisfies $\alpha_1(K) \leq \sqrt{2}$, with equality precisely for a specific family of parallelograms.
In dimension $n=3$, \cref{thm:dbm} yields the analogous inequality $\alpha_j(K) \leq \sqrt{3}$ for $j \in \{1,2\}$ and any point-symmetric $K \in \CK^3$.
According to \cref{thm:dbm_asymm} and \cite[Theorem~1.2]{GrundbacherKobos2024}, the only candidates for $K$ that could achieve equality here are parallelotopes and cross-polytopes.
By \cref{thm:alpha-symmetry}, these two classes of bodies have the same range of values for the $1$st and $2$nd Minkowski chiralities.
Thus, in order to determine whether the bound $\alpha_j(K) \leq \sqrt{3}$ is attained for $j \in \{1,2\}$ and any point-symmetric $K \in \CK^3$, it suffices to check this for just one value of $j$ and either the class of parallelotopes or cross-polytopes.
Note that if the inequality turns out to be attained in particular examples, the subspaces spanned by the principal axes of the John ellipsoid would necessarily be optimal.
In higher dimensions, however, the analogous problem becomes more involved since there does not appear to be an easy description of all point-symmetric $K \in \CK^n$ with $d_{BM}(K,\B_2^n) = \sqrt{n}$ according to \cite[Remark~3.1]{GrundbacherKobos2024}.

Last, it might be helpful to find conditions that describe the optimal situation $K \subset x + \alpha_1(K) \Phi_U(K)$ for $x \in \R^n$ and $U \in \gr{j}{\R^n}$ similar to \cref{thm:brandenberg-koenig}.
Such conditions might help find better general bounds on the Minkowski chirality and understand the optimal reflection subspaces.
Note, though, that \cref{thm:triangle,thm:parallelogram} show that a complete description of the optimal situation for a general convex body might become very involved.

\textbf{Acknowledgements.} KKG and TJ would like to thank Frank Göring for his helpful comments. AC was supported by the German Research Foundation (DFG) Grant CA 3683/1-1. KvD was supported by Postdoc Network Brandenburg (PNB) Individual Grant.

\DeclareFieldFormat{pages}{#1}
\printbibliography[heading=bibintoc]

\bigskip

Andrei Caragea --
Catholic University of Eichstätt--Ingolstadt (KU), Mathematical Institute for Machine Learning and Data Science (MIDS), Germany. \\ \textbf{andrei.caragea@ku.de}

Katherina von Dichter -- 
Brandenburg University of Technology Cottbus--Senftenberg (BTU), Department of Mathematics, Germany. \\
\textbf{katherina.vondichter@b-tu.de}

Kurt Klement Gottwald --
Chemnitz University of Technology, Faculty of Mathematics,  Germany.\\
\textbf{kurt-klement.gottwald@math.tu-chemnitz.de}

Florian Grundbacher -- 
Technical University of Munich, Department of Mathematics, Germany. \\
\textbf{florian.grundbacher@tum.de}

Thomas Jahn --
Friedrich Schiller University Jena, Institute for Mathematics, Germany.\\ \textbf{jahn.thomas@uni-jena.de}

Mia Runge -- 
Technical University of Munich, Department of Mathematics, Germany. \\
\textbf{mia.runge@tum.de}

\vfill\eject

\end{document}